\documentclass[11pt]{article}
\usepackage{amsthm}
\usepackage{amssymb,amsmath,url,color}
\let\pa=\partial
\let\eps=\varepsilon
\let\al=\alpha

\let\f=\frac

\def\na{\nabla}

\def\dv{\mbox{div}}

\def\C{\mathop{\bf C\kern 0pt}\nolimits}
\def\DD{\mathop{\bf D\kern 0pt}\nolimits}
\def\K{\mathop{\bf K\kern 0pt}\nolimits}
\def\N{\mathop{\bf N\kern 0pt}\nolimits}
\def\Q{\mathop{\bf Q\kern 0pt}\nolimits}
\def\R{\mathop{\bf R\kern 0pt}\nolimits}
\def\T{\mathop{\bf T\kern 0pt}\nolimits}

\newcommand{\beq}{\begin{equation}}
\newcommand{\eeq}{\end{equation}}
\newcommand{\ben}{\begin{eqnarray}}
\newcommand{\een}{\end{eqnarray}}
\newcommand{\beno}{\begin{eqnarray*}}
\newcommand{\eeno}{\end{eqnarray*}}

\newtheorem{theorem}{Theorem}[section]
\newtheorem{definition}[theorem]{Definition}
\newtheorem{lemma}[theorem]{Lemma}
\newtheorem{proposition}[theorem]{Proposition}

\begin{document}

\title{The influence of boundary conditions on the contact problem \\ in a 3D Navier-Stokes Flow}
\author{David Gerard-Varet\footnote{Institut de Math\'ematiques de Jussieu, CNRS UMR 7586, Universit\'e Paris 7, 175 rue du Chevaleret, 75013 Paris. Mail: \texttt{gerard-varet@math.jussieu.fr}},\, %
Matthieu Hillairet\footnote{Ceremade, CNRS UMR 7534, Univerisit\'e Paris Dauphine, Place du Mar\'echal de Lattre de Tassigny, 75775 Paris Cedex 16  %
 and Institut Camille Jordan, CNRS UMR 5208, Universit\'e Lyon 1, 43 boulevard du 11 novembre 1918, 69622 Villeurbane Cedex, Mail: \texttt{hillairet@ceremade.dauphine.fr}}\,   %
 and Chao Wang\footnote{Institut de Math\'ematiques de Jussieu, CNRS UMR 7586, Universit\'e Paris 7, 175 rue du Chevaleret, 75013 Paris. Mail: \texttt{wangc@math.jussieu.fr}} %
 }

\date{\today}
\maketitle

\begin{abstract}
We consider the free fall of a sphere above a wall in a viscous incompressible fluid. We investigate the influence of
boundary conditions on the finite-time occurrence of contact between the sphere and the wall. We prove that
slip boundary conditions enable to circumvent the "no-collision" paradox associated with no-slip boundary conditions.
We also examine the case of mixed boundary conditions.
\end{abstract}

\section{Introduction}
Understanding the close-range dynamics of particles inside a fluid is crucial to several fields of application, such as  rheology, sedimentation or slurry erosion. A related problem is to compute the drag induced by nearby solid  bodies: it is important to lubrication theory, or to the control of microswimmers, among others. In the case of a viscous incompressible fluid, a standard way to model the fluid-solid interaction is through the coupling of the Navier-Stokes equations for the fluid part and the conservation of momentum for the bodies.  However, such model reveals unrealistic features: in particular, it can overvalue the drag force induced on  bodies that are close to collision.  Consider for instance the free fall of a rigid sphere over a wall, and assume standard no-slip conditions at the solid surface of the sphere and the wall. Then, the model predicts that no collision is possible between the sphere and the wall, no matter their relative density and the viscosity of the fluid.This is a "no-collision paradox": it has been recognized at first for reduced  models (Stokes approximation, see \cite{Cooley&Oneill69,Cox74}), and more recently for the  full constant-density incompressible Navier Stokes system \cite{DGVH10,Hillairet07b,HillairetTakahashi09}. These results show the need for a more accurate model of the fluid behavior in the "small-distance" regime.  Several refinements were proposed in the past \cite{BarnockyDavis89,davisetal86,KytomaaSchmid92}, see the introduction of \cite{JosephPhD} for a review. In this paper, we investigate the influence of boundary conditions at the fluid-solid interfaces.

\medskip

Boundary conditions are constitutive equations that can be derived theoretically or experimentally. A popular general type of boundary conditions was proposed by Navier
in  the early time of analytical fluid dynamics. This Navier law states that, on the one hand, the normal component of the velocity is continuous
at the fluid/solid interface (impermeability condition) and that, on the other hand, the amount of slip in the tangential part of the velocity is proportional to the
tangential part of the normal stress exerted by the fluid on the boundary. The proportionality coefficient is called "slip length". Assuming the slip length vanishes (\emph{i.e.}, assuming no-slip boundary conditions) suffices to make theoretical and experimental results to coincide in numerous cases. However, developments such as the "no-collision paradox" lead to question this assumption  \cite{Hocking}.
We refer the reader to \cite{Bocquet,Lauga} for a further description of experimental and theoretical validations of the Navier boundary conditions with slip.
{\em We propose here new insights on the contact-problem, by considering the effect of Navier boundary conditions. We shall study  this effect in  the context of the full Navier-Stokes system, which is, up to our knowledge, new.}

\medskip

We consider a  simplified case where only one homogeneous solid is moving in the fluid.
We denote $S(t)\subseteq\mathbb{R}^3, F(t)\subseteq\mathbb{R}^3$ the solid
and the fluid domains at time $t$, and $\Omega\triangleq \overline{S(t)}\cup F(t)$ the total domain. We assume that
the fluid is governed by the constant-density incompressible Navier-Stokes equations :
\begin{eqnarray}\rho_F(\pa_t u_F+u_F\cdot\nabla u_F)-\mu_F\Delta u_F &=& -\nabla p_F-\rho_Fge_3, \phantom{0}  t>0, \quad x\in F(t), \label{F1}\\
\dv \ u_F&=& 0,   \phantom{-\nabla p_F-\rho_Fge_3} t>0,\quad x\in F(t), \label{F2}
\end{eqnarray}
where $u_F$ and $p_F$ stand respectively for the velocity and internal pressure, $\rho_F$ the density,
$\mu_F$ the viscosity in the fluid. We emphasize that we consider here the influence of gravity whose direction
is the third vector of the basis $(e_1,e_2,e_3),$ \emph{i.e.}, the "vertical" unit vector.
The solid motion is governed by Newton laws.
Denoting $x_S(t)\in \mathbb{R}^3$ the position of the center of mass of the solid
at time $t$, $U_S(t)\in \mathbb{R}^3$ its velocity and $\omega_S(t)\in \mathbb{R}^3$ the angular velocity, the Newton laws read
\begin{eqnarray}
m_S\dfrac{d}{dt}U_S(t)&=&-\displaystyle{\int_{\pa S(t)}} (2\mu_FD(u_F)-p_FI)n d\sigma- m_Sge_3, \label{S1}\\
\dfrac{d}{dt}(J_S \omega_S(t))&=&-\displaystyle{\int_{\pa S(t)}} (x-x_S(t)) \times (2\mu_FD(u_F)-p_FI)n   d\sigma. \label{S2}
\end{eqnarray}
Here $\rho_S$ and $m_S\triangleq \rho_S|S(0)|$ are the density and the mass of the solid respectively. Symbol $n$  stands for the unit normal vector to $\partial S(t)$ pointing outside the fluid domain. The 3x3 inertia matrix $J_S(t)$ is defined by
\beno
J_S(t)\triangleq \rho_S\int_{S(t)}(|x-x_S(t)|^2I_d-(x-x_S(t))\otimes(x-x_S(t)))dx.
\eeno
We emphasize that the solid is homogeneous, so that  no angular momentum  is induced by gravity in \eqref{S2}.
Given $(U_S,\omega_S,x_S),$ the solid velocity-field $u_S$ at each point $x \in S(t)$ reads
\begin{equation} \label{eq_solidvel}
u_S(t,x)\triangleq U_S(t)+\omega_S\times(x-x_S(t)).
\end{equation}
In this paper, we complete the system with Navier boundary conditions :
\begin{eqnarray}
(u_F-u_S)\cdot n|_{\pa S(t)}=0,&&\hspace{-42pt} \phantom{u_F\times n|_{\pa\Omega}}(u_F-u_S)\times n|_{\pa S(t)}=-2\beta_S(D(u_F)n)\times n, \label{bc1}\\
u_F \cdot n|_{\pa\Omega}=0,   && \hspace{-42pt}\phantom{(u_F-u_S)\times n|_{\pa S(t)}}u_F\times n|_{\pa\Omega}=-2\beta_{\Omega}(D(u_F)n)\times n, \label{bc2}
\end{eqnarray}
where $\beta_S, \beta_\Omega\geq 0$ are the slip lengths at the fluid-solid and fluid-container interface respectively.
We remind that assuming that $\beta_S$ or $\beta_{\Omega}$ vanishes means assuming no-slip boundary
conditions at the corresponding interface.

\medskip

System \eqref{F1}--\eqref{bc2} is completed with initial conditions :
\begin{eqnarray}
S(0) = S_0, && u_F(0,x) = u_{F,0}(x), \quad x \in \Omega \setminus \bar{S}_0 \, \label{ic1}\\
U(0) = U_0, && \omega(0) = \omega_0\,. \label{ic2}
\end{eqnarray}
Equation \eqref{ic2} may also be written :
\addtocounter{equation}{-1}
\begin{equation}
u_S(0,x) = u_{S,0}(x), \quad  x \in S_0\,.
\end{equation}
 where $u_{S,0}$ is computed with respect to $U_0,\omega_0$ and $S_0$ {\em via} \eqref{eq_solidvel}.
When $S_0 \subset \Omega$ and  $u_{F,0}$ has finite energy, namely $u_{F,0} \in L^2(\Omega \setminus \bar{S}_0),$ standard  computations
show that a reasonable solution to \eqref{F1}--\eqref{ic2} should satisfy the energy estimate :
\begin{multline} \label{nrj_intro}
\dfrac{1}{2} \left[ \int_{F(t)} \rho_F |u_F(t)|^2 + \int_{S(t)} \rho_S |u_S(t)|^2 \right]   \\
+  \mu_F \left(  2  \int_{0}^t \int_{F(t)} |D(u_F(t))|^2 + \dfrac{1}{\beta_{\Omega}} \int_{0}^t \int_{\partial \Omega} |u_F \times n|^2 {\rm d}\sigma  + \dfrac{1}{\beta_S} \int_0^t \int_{\partial S(t)} |(u_F - u_S) \times n |^2 {\rm d}\sigma \right)
\\
\leq
- \int_0^t \int_{F(t)} \rho_F g e_3 \cdot u_F - \int_0^t \int_{S(t)} \rho_S g e_3 \cdot u_S +\dfrac{1}{2} \int_{F_0} \rho_F |u_{F,0}|^2 +\dfrac{1}{2}\int_{S_0} \rho_S |u_S(0)|^2\,.
\end{multline}
where $F_0 \triangleq \Omega \setminus \bar{S}_0.$
When $\beta_S = 0$ or $\beta_\Omega = 0$,  the corresponding boundary term vanishes.

\medskip
Roughly, with regards to the classical theory of  the  Navier-Stokes equation in a fixed domain, one can expect two kinds of solutions for \eqref{F1}--\eqref{ic2}:
\begin{itemize}
\item Strong solutions, locally in time: they should be unique, with possible blow-up  prior to collision in 3D.
\item Weak solutions: they should exist at least up to collision, and be possibly non-unique.
\end{itemize}
In  the no-slip case  ($\beta_S= \beta_{\Omega} =0$), both kinds of solutions have been built in various contexts: 2D/3D problems, bounded/unbounded
container, smooth/singular shape of the solid bodies. Existence of strong solutions locally in time (prior to collision) is now well-known, see \cite{Desjardins&Esteban99,GaldiSilvestre05,DGVH10,Grandmont&Maday98,Takahashi03bis,Takahashi03}. The existence of weak solutions up to collision has been studied in
\cite{Conca&al00,CumsilleTakahashi,Desjardins&Esteban00,Gunzburger&Lee&Seregin00,Hoffmann&Starovoitov99,Hoffmann&Starovoitov00,Serre87,Takahashi&Tucsnak04}. Uniqueness of weak solutions up to collision in 2D is established in the recent paper \cite{GlassSueurpp}.  Finally , extension  of weak solutions after contact is given  in \cite{SanMartin&al02} for 2D problems
and \cite{Feireisl03} for 3D problems.

\medskip
In the case of slip conditions, when $\beta_S >0$ and $\beta_{\Omega} >0$, much less in known.  It is partly because genuine new mathematical difficulties arise: for instance, at the level of weak solutions, discontinuity of the tangential velocity at the fluid/solid interface forbids global $H^1$ bound on the extended velocity-field $u$ as defined by:
\begin{equation} \label{intro_trick}
u = u_S  \text{ in $S(t)$}\,, \qquad u= u_F \text{ in $F(t)$}.
\end{equation}
Such global bound is a key point in the treatment of the Dirichlet case. Hence, to investigate the contact problem for system  \eqref{F1}--\eqref{ic2}, we shall rely on the recent paper \cite{DGVHpp}, which  shows the existence of weak solutions up to collision. Both the definition of weak solutions and the existence result given in  \cite{DGVHpp} will be recalled in the next section.  Let us also mention the work  \cite{SueurPlanaspp}, related to a single body without container.

\medskip
{\em Considering a special class of weak solutions, we will show that slip conditions allow collision in finite time}. Hence, taking into account slight slip at the solid boundaries allows to clear the no-collision paradox. More precisely, we shall consider configurations obeying the following assumptions:
\begin{enumerate}
\item[A1.] The solid $S(t) \sim S_0$ is a ball of radius $1$.
\item[A2.] $\Omega$ is a smooth convex domain of $ \{ x_3 > 0\}$,  flat near:
$$\pa \Omega \supset\{ x_3 = 0, |x| < 2 \delta \}  \: \mbox{ for some } \: \delta \in (0,\frac{1}{4}).$$
\item[A3.] $S(t)$ is axisymmetric with respect to $\{x_1=x_2=0\}$.
\item[A4.] The only possible contact  between the solid  and the container is at $x=0$. More precisely,
$$ dist\left(S(t), \pa \Omega \setminus \{ x_3 = 0, |x| < \delta \} \right) \ge d_\delta > 0. $$
\end{enumerate}
Parameters $\delta$ and $d_\delta$ will be involved in the proof. Let us comment on these  assumptions. {\em Their only goal is to  ensure that  the geometry of the potential contact zone does not change with time.} Namely, it is an aperture between a sphere and a plane, only varying with the distance $h(t)$ between the south pole of $S(t)$ and $x=0$. The whole point is then to determine if $h(t)$ can vanish in finite time or not. This simplifies greatly our computations. However, as will be clear from our proof, our result is genuinely local in space time: it  is independent from the global characteristics of the solid/container geometry.

\medskip
Note that  A1-A2 are just assumptions on $S_0$ and  $\Omega$. If $S_0 \subset \subset \Omega$, and  if the initial data $(u_{F,0},u_{S,0})$ has finite energy, article \cite{DGVHpp} ensures the existence of a weak solution up to contact between $S(t)$ and $\partial \Omega$ (see next section). Assumption A3 states that the axisymmetry of the solid is preserved with time. Actually, it will be automatically satisfied if $S_0$, $\Omega$ and the initial velocities are axisymmetric.  More precisely, in this case, one can adapt the construction of weak solutions  from \cite{DGVHpp} by including  the symmetry constraint, and obtain directly an axisymmetric weak solution. Finally, assumption A4  guarantees that the sphere remains away from the top and lateral boundaries. Again, this can be ensured by a proper choice of the container and the initial data. As regards lateral boundaries, it is enough by assumption A3 that they are at distance $> 1$ to the  vertical axis $\{x_1 = x_2 = 0\}$. Also,  there are several ways to avoid merging of the solid body and the top boundary. One can for instance assume that the solid is heavier than the fluid ($\rho_S > \rho_F$), and that the initial kinetic energy of the fluid and the solid is small enough. As the total energy of weak solutions is non-increasing, this means that the increase of the potential energy through time can only be small: hence, the solid will not rise too much, and so will remain away from the top boundary. Again, we insist that our method of proof could extend to more general settings.

\medskip
We can now state our main results:
\begin{theorem}[Slip Case] \label{thm_slip}
Assume  $\beta_S, \beta_{\Omega} > 0$ and  $\rho_S > \rho_F$. For any weak solution $(S,u)$  (see Definition  \ref{def:ws_s}) satisfying A1-A4, the solid body $S$  touches  $\pa \Omega$ in finite time.
\end{theorem}

\begin{theorem}[Mixed Case] \label{thm_mixed}
Let $\beta_S$ or $\beta_{\Omega}$ vanish. Any weak solution $(S,u)$  (see Definition  \ref{def:ws_m}) satisfying A1-A4 is global. In particular, the solid body $S$ never touches
$\partial \Omega.$
\end{theorem}
We recall that the full no-slip case (when $\beta_S = \beta_{\Omega} = 0$) was already treated in \cite{HillairetTakahashi09}.
Theorems \ref{thm_slip} and \ref{thm_mixed} show that one possible way  to obtain realistic contact, and to circumvent the no-collision paradox,
 is to introduce slip in the boundary conditions {\em on all solid boundaries of the fluid domain}.

\medskip
We conclude with  a few hints on the proof of Theorem \ref{thm_slip} (ideas for Theorem  \ref{thm_mixed} are similar). The proof elaborates on the paper \cite{DGVH12}, devoted to a simplified linear system:  one  takes there $\Omega = \R^3_+$, while  the fluid outside the sphere $S(t)$ is governed by a steady Stokes flow (with $\mu_F = 1$ for simplicity):
 $$ - \Delta u_F = -\nabla p_F-\rho_Fge_3, \quad \na \cdot u_F = 0, \quad x \in F(t).  $$
In this simplified setting, the fluid and solid domains $F(t) = F_{h(t)}$ and $S(t) = S_h(t)$  are characterized by the distance $h(t)$ beween the sphere and the plane wall. Then, it is easily seen that $u_F(t) = h'(t) w_{h(t)}$ where $w_h$ is the solution of a normalized Stokes equation, set in the domain $F_h$  "frozen"  at distance $h$:
 \begin{equation} \label{wh}
 \left\{
\begin{aligned}
  -\Delta w_h +  \nabla p_h   = 0, & \quad \na \cdot w_h = 0, \quad x \in F_h,  \quad (w_h-e_3)\cdot n\vert_{\pa S_h}=0,  \quad w_h \cdot n|_{\pa\Omega}=0, \\
 (w_h-e_3)\times n\vert_{\pa S_h} &=-2\beta_S(D(w_h)n)\times n, \\
 w_h \times n|_{\pa\Omega} & =-2\beta_{\Omega}(D(w_h)n)\times n.
 \end{aligned}
 \right.
 \end{equation}
 We refer to \cite{DGVH12} for all details. Eventually, the dynamics reduces to  an ODE of the type
  $$ h''(t) = - h'(t) {\cal D}(h(t)) + \frac{\rho_S - \rho_F}{\rho_S} g $$
where the {\em drag term}  ${\cal D}(h)$ is the energy of $w_h$:
$$ {\cal D}(h) \: =  \: {\cal E}_h(w_h), \quad {\cal  E}_h(u) \: :=  \:  \int_{F_h} | \na u|^2 + \left( \frac{1}{\beta_S} + 1\right) \int_{\pa S_h} | (u - e_3) \times n |^2  + \frac{1}{\beta_\Omega}  \int_{\pa \Omega} | u \times n |^2. $$
 For this reduced model, the contact problem resumes to the computation of the drag ${\cal D}_h$ in the limit $h \rightarrow 0$. The difficulty is that  there is no simple formula for the solution $w_h$ of the Stokes system in $F_h$.

 \medskip
 A general approach to this problem was recently proposed in \cite{DGVH12}. It works for no-slip conditions, slip ones, and for more general geometries. The starting idea is to use a variational characterization of the drag: one can identify  ${\cal D}(h)$ as the minimizer of an energy ${\cal E}_h$, over an admissible set of fields ${\cal A}_h$:
 $$ {\cal D}(h) \: = \inf_{u \in {\cal A}_h} {\cal E}_h(u) = {\cal E}_h(w_h). $$
 Then, instead of computing the true minimizer, the point is to find an appropriate relaxed minimization problem, for which  the minimizer  can be easily computed, and such that the corresponding minimum is close to the exact one asymptotically in $h$. The choice of such relaxed problem is developed in various situations in  \cite{DGVH12}.
Roughly,  in the case of slip conditions, one obtains  that $c |\ln (h)| \: \le \:  D(h) \: \le \: C |\ln(h)|$, in agreement with older computations of Hocking \cite{Hocking}. Hence, the drag is weak enough to allow for collision, within this linear Stokes approximation. See also \cite{HLS11} for numerical insight.

 \medskip
 In the present paper, we manage to use ideas of the  linear study, in the context of the nonlinear Navier-Stokes flow. At a formal level, the first idea is to multiply the momentum equation by $w(t,x) = w_{h(t)}(x)$, $w_h$ solution of \eqref{wh} (where this time $\Omega$ is the container). After formal  manipulations, one obtains
 \begin{equation} \label{doth}
  h''(t)  \: = \:  - h' {\cal D}(h)   \: + \frac{\rho_S - \rho_F}{\rho_S} g \: + \: \mbox{remainder}
  \end{equation}
  where the remainder term involves the test function $w$ and the convective derivative of $u$. However, the Stokes solution $w_h$ is not known explicitly. {\em Hence, the main point in our proof is to use another test function, based on the relaxed minimizer introduced in \cite{DGVH12}. Then, one must show that for such test function, the remainder term is indeed small, that is controlled by the gravity when the solid is heavier than the fluid}. Let us stress that several difficulties arise in this process.  For instance, the relaxed minimizer introduced in the linear study is not accurate enough, so that we must improve it. Moreover, due to the slip conditions, the natural quantity is the symmetric gradient of the weak solution, not the full gradient. This makes the derivation of the various bounds more difficult than in the no-slip case.

\medskip

The outline of the paper is as follows. In the next section,  we introduce a suitable notion of weak solutions for system \eqref{F1}--\eqref{ic2},  in the slip case and in the mixed case. We recall the existence theorems up to collision available for such weak solutions. in the slip case and in the mixed case. The two last sections are devoted to the proofs of the two theorems.

\section{On the definition of weak solutions}

We introduce here a definition of weak solutions for \eqref{F1}--\eqref{ic1} in  two cases: slip boundary conditions at all boundaries, mixed slip/no-slip
boundary conditions. We mostly follow \cite{DGVHpp}.  Given $\Omega$ a Lipschitz domain, we set:
\begin{eqnarray*}
\mathcal{D}_{\sigma}(\Omega)& \triangleq&  \{ \varphi\in C^\infty_c(\Omega),\quad \dv \, \varphi=0 \},
     \qquad \phantom{1121}\mathcal{D}_\sigma(\overline\Omega)\triangleq\{ \varphi|_{\Omega}, \: \varphi\in\mathcal{D}_\sigma(\mathbb{R}^3) \},\\[4pt]
L^2_{\sigma}(\Omega) &\triangleq& \text{\rm the closure of } \mathcal{D}_{\sigma}(\Omega) \textrm{ in } L^2(\Omega),
   \qquad H^1_{\sigma}(\Omega) \triangleq H^1(\Omega)\cap L^2_\sigma(\Omega),\\[4pt]
 H^1_{\sigma}(\overline\Omega) &\triangleq & \textrm{ the closure of } \mathcal{D}_{\sigma}(\overline\Omega) \textrm{ in }  H^1(\Omega),\\[4pt]
 \mathcal{R} &\triangleq &\{\varphi_s,\quad \varphi_s=V+\omega\times x,\quad\textrm{for some}\quad V\in\mathbb{R}^3, \omega\in \mathbb{R}^3 \},
\end{eqnarray*}

\subsection{The case of slip boundary conditions: $\beta_S >0$ and $\beta_{\Omega} >0$.}
We consider here that the slip lengths $\beta_S$ and $\beta_{\Omega}$ do not vanish \emph{i.e.}, that the fluid slips on both boundaries.
The corresponding Cauchy problem is studied in \cite{DGVHpp}, where the following definition of weak solution is given:
\begin{definition} \label{def:ws_s}
Let $\Omega$ and $S_0\Subset\Omega$ be two Lipschitz bounded domains of $\mathbb{R}^3$. Let $u_{F,0}\in L^2_{\sigma}(\Omega), u_{S,0}\in \mathcal{R}$ such that
$u_{F,0}\cdot n=u_{S,0}\cdot n$ on $\pa S_0$. Let $T > 0$, or $T =\infty$.

\smallskip
 A weak solution of (\ref{F1})--(\ref{ic2}) on $[0,T)$ is a pair $(S, u)$ satisfying

 \smallskip
 \noindent
1) $S(t)\Subset \Omega$ is a bounded domain of $\mathbb{R}^3$ for all $t\in[0,T)$ such that
\beno
\chi_S(t,x)\triangleq 1_{S(t)}(x)\in L^\infty((0,T)\times\Omega).
\eeno
2) $u$ belongs to the space
\begin{multline*}
\mathcal{S}_T\triangleq \{ u\in L^\infty(0,T;L^2_\sigma(\Omega)),  \textrm{ there exists } u_F \in L^2_{loc}([0,T); H^1_\sigma(\Omega)), \: u_S\in L^\infty(0,T; \mathcal{R})  \\
  \phantom{1qdfZERTYUT}  \textrm{ such that }  \:  u(t,\cdot)=u_F(t,\cdot) \textrm{ on } F(t),    \quad u(t,\cdot)=u_S(t,\cdot) \textrm{ on } S(t),\, \textrm{for all } t\in[0,T)\}
\end{multline*}
where $F(t) \triangleq \Omega \setminus \overline{S(t)}$. \\
3) For all $\varphi$ belonging to the space
\begin{multline*}
\mathcal{T}^{s}_{T} \triangleq \{\varphi\in C([0,T];L^2_{\sigma}(\Omega)),
\textrm{ there exists } \varphi_F\in \mathcal{D}([0,T);\mathcal{D}_{\sigma}(\overline\Omega)),
   \varphi_S\in \mathcal{D}([0,T);\mathcal{R})\\
  \phantom{1qdfZERTYUT}\textrm{ such that } \: \varphi(t,\cdot)=\varphi_F(t,\cdot) \, \textrm{ on }\, F(t),  \quad \varphi(t,\cdot)=\varphi_S(t,\cdot)\,\textrm{ on }\, S(t),\, \textrm{for all } t\in[0,T]\}\,
\end{multline*}
there holds
\begin{eqnarray}\label{ws_s}
&&-\int_0^T\int_{F(t)}\rho_F u_F\cdot\pa_t\varphi_F - \int_0^T\int_{S(t)}\rho_S u_S\cdot\pa_t\varphi_S - \int_0^T\int_{F(t)}\rho_Fu_F\otimes u_F: \nabla \varphi_F
 \nonumber\\
&& + \f{\mu_F}{\beta_\Omega}\int_0^T\int_{\pa\Omega}(u_F\times n)\cdot(\varphi_F\times n)
+ \f{\mu_F}{\beta_S}\int_0^T\int_{\pa S(t)}((u_F - u_S) \times n)\cdot ((\varphi_F - \varphi_S)\times n)  \nonumber\\
&&+  \int_0^T\int_{F(t)}2\mu_FD(u_F):D(\varphi_F)  \quad = \quad -\int_0^T\int_{F(t)}\rho_F g e_3 \varphi_F\nonumber -\int_0^T\int_{S(t)}\rho_S g e_3 \varphi_S\nonumber\\
&& + \int_{F_0}\rho_F u_{F,0}\varphi_F(0,\cdot) + \int_{S_0}\rho_S u_{S,0}\varphi_S(0,\cdot),
\end{eqnarray}
4) For all $\psi \in \mathcal{D}([0,T);\mathcal{D}(\overline\Omega))$, we have
\beno
-\int^T_0\int_{S(t)}\pa_t\psi-\int^T_0\int_{S(t)}u_S\cdot \nabla \psi=\int_{S_0}\psi(0,\cdot),
\eeno
5) We have the following energy inequality for a.a. $t \in (0,T)$ :
\begin{multline} \label{nrj_est_s}
\dfrac{1}{2} \left[ \int_{F(t)} \rho_F |u_F(t)|^2 + \int_{S(t)} \rho_S |u_S(t)|^2 \right] \\
+  \mu_F \left(  2  \int_{0}^t \int_{F(t)} |D(u_F(t))|^2 + \dfrac{1}{\beta_{\Omega}} \int_{0}^t \int_{\partial \Omega} |u_F \times n|^2 {\rm d}\sigma
+\dfrac{1}{\beta_S} \int_0^t \int_{\partial S(t)} |(u_F - u_S) \times n |^2 {\rm d}\sigma \right) \\
\leq
- \int_0^t \int_{F(t)} \rho_F g e_3 \cdot u_F - \int_0^t \int_{S(t)} \rho_S g e_3 \cdot u_S + \dfrac{1}{2} \int_{F(0)} \rho_F |u_{F,0}|^2 + \dfrac{1}{2} \int_{S(0)} \rho_S |u_S(0)|^2\,.
\end{multline}
\end{definition}
We refer to \cite{DGVHpp}  for a detailed discussion of this definition.  Note that it is restricted to solutions up to collision, through the condition $S(t) \Subset \Omega$. We recall that item {\em 3)} is the identity one gets by multiplying \eqref{F1} by a divergence-free test-function $\varphi_F$ and then combining
with \eqref{S1}--\eqref{S2}. Item {\em 4)} is the weak form of the transport equation satisfied
by $\chi_S$, the indicator function of $S(t)$ as defined in {\em 1)}. Finally, {\em 2)} summarizes the regularity that is implied by the energy estimate \eqref{nrj_est_s}. Note that $u_F$ and $u_S$ share the same normal component, so that the $L^\infty L^2$ bound on $u_F$ and $u_S$ yields a global $L^\infty L^2_\sigma(\Omega)$ on  the extended velocity field $u$ (see \eqref{intro_trick} for the definition of $u$).
Moreover, the $L^\infty L^2$ bound on $u_F$ and $L^2 L^2$ bound on  $D(u_F)$ imply an $L^2_{loc}H^1$ bound, through Korn inequality (that is valid as long as the solid is away from the boundary of the container).

\medskip
From \cite{DGVHpp}, there is a weak solution for the system up to a contact between $S(t)$ and $\partial \Omega$:
\begin{theorem}\label{thm:ws_s}
Let  $\Omega$ and $S_0 \Subset \Omega$ two $C^{1,1}$ bounded domain of $\mathbb{R}^3$.
Given $(\beta_S,\beta_{\Omega}) \in (0,\infty)^2,$ and initial data $u_{F,0}\in L^2_{\sigma}(\Omega), u_{S,0}\in \mathcal{R}$ satisfying:
$$u_{F,0}\cdot n= u_{S,0}\cdot n,\quad \textrm{on} \quad \pa S_{0},$$
there exists  $T \in \R^*_+ \cup \{\infty\}$ and a weak solution of (\ref{F1})-(\ref{ic2}) on $[0,T)$. Moreover, such  weak solution exists up to collision, that is either we can take $T=\infty$, or we can take $T \in \R_+^*$ in such a way that
\beno
S(t)\Subset \Omega,\quad\textrm{for all}\quad t\in[0,T),\quad \textrm{and}\quad \lim_{t\rightarrow T^{-}}\textrm{dist}(\pa S(t),\pa\Omega)=0.
\eeno
\end{theorem}
\noindent The aim of {Theorem \ref{thm_slip}} is to prove that $T < \infty$ under assumption A1-A4.

\subsection{The mixed case: $\beta_S = 0$ and $\beta_{\Omega} >0$.}
We proceed with the case where slip is imposed on one of the boundaries only. For simplicity, we consider that no-slip boundary conditions are imposed on $\partial S(t)$
\emph{i.e.} $\beta_S = 0$ and $\beta_{\Omega} >0.$ To our knowledge, this particular system has not been  treated explicitly in previous studies. Nevertheless, as slip is allowed on the exterior boundary of $F(t)$ only, it is straightforward to adapt the method of \cite{Desjardins&Esteban00} or \cite{Hoffmann&Starovoitov00}, for instance, to tackle a Cauchy theory for weak solutions before contact. In particular, the extension trick \eqref{intro_trick} yields a sufficiently smooth velocity-field in this case  so that we can reduce the full system to a global weak formulation on this  extended velocity field. This reads as follows:
\begin{definition}   \label{def:ws_m}
Let $\Omega$ and $S_0\subset\Omega$ be two Lipschitz bounded domains of $\mathbb{R}^3$. Let $u_{F,0}\in L^2_{\sigma}(\Omega), u_{S,0}\in \mathcal{R}$ such that
$u_{F,0}\cdot n=u_{S,0}\cdot n$ on $\pa S_0$. A weak solution of (\ref{F1})--(\ref{ic2}) is a pair $(S, u)$ satisfying

\smallskip
 \noindent
1) $S(t)\Subset \Omega$ is a bounded domain of $\mathbb{R}^3$ for all $t\in[0,T)$ such that
\beno
\chi_S(t,x)\triangleq 1_{S(t)}(x)\in L^\infty((0,T)\times\Omega).
\eeno
2) $u \in L^{\infty}(0,T;L^2_{\sigma}(\Omega)) \cap L^2(0,T;H^1_{\sigma}(\Omega))$ satisfies
$$
u(t,\cdot)_{|_{S(t)}} \in \mathcal R \quad \text{ for a.a. $t \in (0,T)\,.$}
$$
3) For all $\varphi$ belonging to
$$
\mathcal{T}^{m}_{T} \triangleq  \{\varphi\in \mathcal D([0,T) \times \bar \Omega), \quad \varphi(t,\cdot)_{|_{S(t)}} \in \mathcal R \text{ for all $t \in [0,T)$}\}
$$
there holds
\begin{multline}\label{ws_m}
-\int_0^T\int_{\Omega}\rho \left( u \pa_t\varphi  +  u \otimes u : D(\varphi) \right)  +\int_0^T\int_{\Omega}2\mu_F D(u):D(\varphi) \\
 + \f{1}{\beta_\Omega}\int_0^T\int_{\pa\Omega}(u \times n)\cdot(\varphi \times n)
  =-\int_0^T\int_{\Omega}\rho  g e_3 \varphi -\int_{\Omega}\rho_0 u_0 \varphi(0,\cdot),
\end{multline}
where
\begin{eqnarray*}
\rho(t,x) &=& \rho_S \, 1_{S(t)}(x) + \rho_F 1_{F(t)}(x), \quad \text{ for $(t,x) \in (0,T) \times \Omega,$}\\
\rho_0(x)& = &\rho_S \, 1_{S_0}(x) + \rho_F 1_{F_0}(x),  \quad \phantom{.22} \text{for $x \in \Omega,$} \\
 u_{0}(x) &=& u_{S,0} \, 1_{S_0}(x) + u_{F,0} 1_{F_0}(x),  \quad \phantom{.22} \text{for $x \in \Omega,$}
\end{eqnarray*}
4) For all $\psi \in \mathcal{D}([0,T);\mathcal{D}(\overline\Omega))$, we have
\beno
-\int^T_0\int_{S(t)}\pa_t\psi -\int^T_0\int_{S(t)}u \cdot \nabla \psi=\int_{S_0}\psi(0,\cdot),
\eeno
5) We have the following energy inequality for a.a. $t \in (0,T)$ :
\begin{multline} \label{nrj_est_m}
\dfrac{1}{2}  \int_{\Omega} \rho |u(t)|^2 + 2 \mu_F \int_{0}^t \int_{\Omega} |D(u)|^2  + \dfrac{1}{\beta_{\Omega}} \int_{0}^t \int_{\partial \Omega} |u  \times n|^2 {\rm d}\sigma  \\
\leq
- \int_0^t \int_{F(t)} \rho g e_3  u  + \dfrac{1}{2}\int_{\Omega} \rho(0,\cdot) |u_{0}|^2 \,.
\end{multline}
\end{definition}
Following the method of proof of \cite{Desjardins&Esteban00}, one obtains:

\begin{theorem}\label{thm:ws_m}
Let  $\Omega$ and $S_0\Subset \Omega$ two $C^{1,1}$ bounded domain of $\mathbb{R}^3$.
Given $\beta_{\Omega} \in (0,\infty)$ and initial data $u_{F,0}\in L^2_{\sigma}(\Omega), u_{S,0}\in \mathcal{R}$ satisfying
$$  u_{F,0}\cdot n= u_{S,0}\cdot n,\quad \textrm{on} \quad \pa S_{0}, $$
the same conclusion as in Theorem \ref{thm:ws_s} holds.
\end{theorem}

\section{Proof of collision with slip conditions} \label{sec_slip}
This section is devoted to the proof of Theorem \ref{thm_slip}.  We assume $\beta_S$ and $\beta_{\Omega}$ are fixed non-negative slip lengths, and  $\rho_S > \rho_F$. Let $(S,u)$ be a weak solution over $(0,T)$,  as given by Theorem \ref{thm:ws_s}. We consider a solution up to its "maximal time of existence": that is $T=\infty$ if there is no collision,  or $T < \infty$ the time of collision.
We  assume that this weak solution satisfies A1-A4 for $t \in [0,T)$.   We denote by $h=h(t)$ the distance between the south pole of $S(t)$  and $x=0$.  By assumption A4,   when $h(t)$ is small enough,  it is equal to the distance between $S(t)$ and the boundary $\partial \Omega$. Our goal is to prove that $h(t)$ goes to zero in finite time. In other words, $T < \infty$, that is collision occurs in finite time.

\medskip
To build up  our proof, we first need some general considerations on the connection between  the weak formulation \eqref{ws_s} and  the distance $h$,  for an arbitrary weak solution satisfying A4. We then design a special  test-function in order to prove  contact in finite time.

\subsection{Some features of the weak formulation \eqref{ws_s}}
Let $(S,u)$ a weak solution defined on $(0,T)$ and $\varphi \in \mathcal{T}_T^s$ of the form
\begin{equation} \label{eq_varphi}
\varphi  = \zeta(t) \varphi_{h(t)} \text{ with } \qquad  \varphi_S = \zeta(t)e_3\,.
\end{equation}
for some functions $\zeta \in \mathcal D([0,T))$ and $h \mapsto \varphi_{h}$ to be made precise.
Integrating by parts, we get that for any $q \in \mathcal D([0,T) \times \bar{\Omega}),$
for a.a. $t \in (0,T)$
\beno
&&2 \int_{F(t)} D(\varphi_F) : D(u_F)\\
&=& \int_{F(t)} \left( 2 D(\varphi_F) - q I \right) : \nabla u_F\\
&=&-\int_{F(t)}(\Delta \varphi_F-\nabla q)\cdot u_F+\int_{\pa F(t)}(2D(\varphi_F)-qI)n\cdot u_F d\sigma\\
&=&-\int_{F(t)}(\Delta \varphi_F-\nabla q)\cdot u_F+\int_{\pa \Omega}(2D(\varphi_F)-qI)n\cdot u_F d\sigma+\int_{\pa S(t)}(2D(\varphi_F)-qI)n\cdot u_F d\sigma\\
&=&-\int_{F(t)}(\Delta \varphi_F-\nabla q)\cdot u_F+\int_{\pa \Omega}(2 D(\varphi_F)-qI)n\cdot u_F d\sigma\\
&&\quad+\int_{\pa S(t)}(2D(\varphi_F)-qI)n\cdot (u_F-u_S) d\sigma+\int_{\pa S(t)}(2D(\varphi_F)-qI)n\cdot u_S d\sigma,
\eeno
Chosing also $q$ of the form $q = \zeta(t) q_{h(t)},$
the last term can be calculated by
\beno
\int_{\pa S(t)}(2D(\varphi_F)-qI)n\cdot u_S d\sigma=h'(t) \zeta(t)\int_{\pa S(t)}(2D(\varphi_{h(t)})-q_{h(t)}I)n\cdot e_3 d\sigma \triangleq \zeta(t) h'(t) n(h(t)).
\eeno
We emphasize that the domain $S(t)$ is completely fixed by the value of $h(t)$ (we shall denote $S_{h(t)}$ in what follows)
so that the last integral does depend on $h(t)$ only.
Combining all the equations with (\ref{ws_s}), we find
\ben\label{drag}
&&\mu_F \int_0^T \zeta(t) h'(t)n(h(t)) \text{d$t$}\nonumber\\
&=&\int_0^T\int_{F(t)}\rho_F u_F\cdot\pa_t\varphi_F +  \int_0^T\int_{S(t)}\rho_S u_S\cdot\pa_t\varphi_S + \int_0^T\int_{F(t)}\rho_Fu_F\otimes u_F:\nabla \varphi_F\nonumber\\
& -&\int_0^T\int_{F(t)}\rho_F ge_3\varphi_F -\int_0^T\int_{S(t)}\rho_S g e_3 \varphi_S + \int_{F_0}\rho_F u_{F,0}\varphi_F(0,\cdot) + \int_{S_0}\rho_S u_{S,0}\varphi_S(0,\cdot)\nonumber\\
&-&\f{\mu_F}{\beta_\Omega}\int_0^T\int_{P}(u_F\times n)\cdot(\varphi_F\times n)-\f{\mu_F}{\beta_S}\int_0^T\int_{\pa S(t)}((u_F-u_S)\times n)\cdot((\varphi_F-\varphi_S)\times n)\nonumber\\
&+& \mu_F \Bigg( \int_0^T\int_{F(t)}(\Delta \varphi_F-\nabla q)\cdot u_F - \int_0^T\int_{\partial \Omega}(2 D(\varphi_F)-qI)n\cdot u_F d\sigma\nonumber\\
&-&\int_0^T\int_{\pa S(t)}(2 D(\varphi_F)-qI)n\cdot (u_F-u_S) d\sigma \Bigg)
\een
This identity enables to describe the time evolution of $h$ through the function $n$ that we have to choose now.
To compute $n(h),$ we shall take advantage of the following identity:
\begin{eqnarray}
n(h)
&=&\int_{\pa S_h}(2D (\varphi_{h})-q_{h}I)n\cdot \varphi_{h} d\sigma+\int_{\pa S_h}(D(\varphi_{h})-q_{h}I)n\cdot (e_3-\varphi_{h}) d\sigma \notag\\
&=& \int_{F_h}(\Delta \varphi_{h}-\nabla q_{h}) \cdot \varphi_{h} dx+ 2 \int_{F_h}D(\varphi_{h}) : D(\varphi_{h}) dx \notag \\
& - & \int_{\pa \Omega}( 2D(\varphi_{h})-q_{h}I)n\cdot \varphi_{h} d\sigma+\int_{\pa S_h}(D(\varphi_{h})-q_{h}I)n\cdot (e_3-\varphi_{h}) d\sigma. \label{n(h)}
\end{eqnarray}
where $F_h \triangleq \Omega \setminus \bar{S}_h.$ This identity holds true for all $h$ such that $S_h \subset\subset \Omega.$

\subsection{Construction of test-functions} \label{sec_tf}
We construct now the test function $\varphi_h$ and the pressure $q_h$ associated with
the solid particle $S_h$  frozen at distance $h$.  It will be defined for  $h \in (0, h_M)$, with $h_M \triangleq \sup_{t \in [0,T)} h(t)$.
When $h \to 0,$ a cusp arises in $F_h = \Omega \setminus \overline{S_h}.$ This cusp is contained in the domain :
$$
\Omega_{h,r} := \{ (x_1,x_2,x_3) \in F_h,  \: x_3 < \frac{1}{2} +h  \text{ and } |(x_1,x_2)| < r\}\,,
$$
for arbitrary $r \in (0,\frac{1}{2}).$ In particular, in all forthcoming estimates, we shall pay special attention to the region  $\Omega_{h,2 \delta}$, where $\delta$ was introduced in assumption A2.

\medskip

To compute test-functions and associated pressures, we introduce cylindrical coordinates $(r,\theta,z)$ associated with cartesian
coordinates $(x_1,x_2,x_3)$ in such a way that $x_3=z.$ The local basis reads $(e_r,e_{\theta},e_z).$
First, we introduce an approximate solution of the steady Stokes equation in the domain $\Omega_{h,2\delta}$. It will satisfy locally the  slip boundary conditions \eqref{bc1}-\eqref{bc2},  with $u_S$ replaced by $e_z$:
$$\widetilde{\varphi}_h \\
=\f{1}{2}\left(-\pa_z\Phi\left(r,\f{z}{h+\gamma_s(r)}\right)x_1, -\pa_z\Phi\left(r,\f{z}{h+\gamma_s(r)}\right)x_2,\f{1}{r}\pa_r\left(\Phi\left(r,\f{z}{h+\gamma_s(r)}\right)r^2\right)\right)
$$
where
\beq \label{def_Phideb}
\gamma_s(r)=1-\sqrt{1-r^2},  \quad \Phi(r,t)=P_1(r)t+P_2(r)t^2+P_3t^3.
\eeq
Here
\begin{eqnarray}
P_1(r)&=&\f{6(2+\al_S)}{12+4(\al_S+\al_P)+\al_S\al_P},\\[4pt]
 P_2(r) &= & \f{3(2+\al_S)\al_P}{12+4(\al_S+\al_P)+\al_S\al_P},\\[4pt]
 P_3(r)&=&-\f{2(\al_S+\al_S\al_P+\al_P)}{12+4(\al_S+\al_P)+\al_S\al_P},
\end{eqnarray}
where
\beq \label{def_Phifin}
\al_P=\f{h+\gamma_s(r)}{\beta_\Omega},\quad \al_S= \left( \f{1}{\beta_S}  + 2\right)   (h+\gamma_s(r)).
\eeq
This approximation is inspired by the computations in \cite{DGVH12}.
We emphasize that the value of $\al_S$ does not match exactly the one in \cite{DGVH12}.
This modification is required so that further boundary estimates hold true (see \eqref{est:bdy}). {\em We stress that  notation $\pa^k_z \Phi(r, \frac{z}{h+\gamma_s(r)})$, resp.  $\pa^k_r \Phi(r, \frac{z}{h+\gamma_s(r)})$, stands for
$\displaystyle  \pa_z^k\Psi(r,z)$, resp.   $\displaystyle \pa_r^k\Psi(r,z), \quad   \Psi(r,z)  \triangleq     \Phi(r, \frac{z}{h+\gamma_s(r)}).$}
With this structure, $\tilde \varphi_h$ is divergence-free. Moreover, the function $\Phi$ is constructed so that
it satisfies the following boundary conditions:
\begin{itemize}
\item for $z=0$
\begin{eqnarray}
&&\Phi(r,0) = 0, \\
&& \partial_{zz} \Phi\left(r,\f{z}{h+\gamma_s(r)}\right) - \dfrac{1}{\beta_{\Omega}} \partial_{z} \Phi\left(r,\f{z}{h+\gamma_s(r)}\right) = 0
\end{eqnarray}
\item for $z= h + \gamma_S(r)$
\end{itemize}
\begin{eqnarray}
&&  \Phi(r,1) = 1\, \\
&& \partial_{zz} \Phi \left(r,\f{z}{h+\gamma_s(r)}\right)+\left( \dfrac{1}{\beta_{S}} + 2 \right)  \partial_{z} \Phi\left(r,\f{z}{h+\gamma_s(r)}\right) = 0 \label{prop_Philast}
\end{eqnarray}
In the local cylindrical basis, the approximate solution $\widetilde{\varphi}_h$ reads:
\begin{equation} \label{tildephih}
\widetilde{\varphi}_h= \f{1}{2} \left[ -\pa_z\Phi\left(r,\f{z}{h+\gamma_s(r)}\right)r e_r +\f{1}{r}\pa_r\left(\Phi\left(r,\f{z}{h+\gamma_s(r)}\right)r^2\right) e_z  \right]\,.
\end{equation}
In particular, if we remark that $n$ satisfies $n = -r e_r + (1+h-z) e_z$ on $\partial S_h \cap \Omega_{h,1/2}$ and  $n = e_3$ on $\partial \Omega \cap \Omega_{h,1/2}$, we obtain (see the introduction of Appendix \ref{app} for further details) :
\ben\label{bdy1}
\left\{
\begin{array}{rcl}
(\widetilde{\varphi}_h-e_3)\cdot n|_{\pa S_h \, \cap \, \partial \Omega_{h,2\delta}}&=&0,\\[4pt]
\widetilde{\varphi}_h\cdot n|_{\pa\Omega \, \cap \, \partial \Omega_{h,2\delta}}&=&0,\\[4pt]
\widetilde{\varphi}_h\times n|_{\pa\Omega \, \cap \, \partial \Omega_{h,2\delta}}&=&-2\beta_{\Omega}(D(\widetilde{\varphi}_h)n)\times n|_{\pa\Omega \, \cap \, \partial \Omega_{h,2\delta}}.
\end{array}
\right.
\een
Thus, up to its extension  to the whole  of $\Omega$,  the function $\widetilde{\varphi}_h$ is a good candidate for applying the computations of the previous section. Such extension  $\varphi_h$ should satisfy $\varphi_h\vert_{S_h} = e_z$. It remains to define it into the whole of $F_h$.

\medskip
To this end, we introduce a smooth  function $\phi=\phi(x)$ with compact support satisfying
\beno
\begin{array}{rcl}
\phi=1 && \mbox{ on  a $\frac{d_\delta}{2}$-neighborhood of the unit ball $B(0,1)$ },\\[4pt]
\phi = 0 &&   \mbox{ outside a $d_\delta$-neighborhood of $B(0,1)$}.
\end{array}
\eeno
We recall that $d_\delta$ was introduced in A4.
We also define a cut-off function $\chi : \mathbb R^3 \to [0,1]$ s.t.
\begin{equation} \label{eq_chi}
\chi=1  \textrm{ in } (-\delta,\delta)^3,\qquad \chi=0 \textrm{ outside }  (-2\delta,2\delta)^3.
\end{equation}
With these cut-off functions, we set (in cartesian coordinates):
\begin{eqnarray}\label{test_s}
&&
\varphi_h =\f{1}{2}\left(
\begin{array}{c}
  -x_1\pa_z\left[(1-\chi(x))  \phi\left(x - (1+h) e_3\right) \, + \, \chi(x) \Phi(r,\f{z}{h+\gamma_s(r)})\right]\\[6pt]
   -x_2\pa_z\left[(1-\chi(x))  \phi\left(x- (1+h) e_3\right) \, + \, \chi(x) \Phi(r,\f{z}{h+\gamma_s(r)})\right] \\[6pt]
  \f{1}{r}\pa_r\left[r^2((1-\chi(x))  \phi\left(x-(1+h) e_3\right) \, + \, \chi(x) \Phi(r,\f{z}{h+\gamma_s(r)}))\right]
    \end{array}
\right) \textrm{ in } F_h,\nonumber\\
&&
\varphi_h=e_3 \textrm{ in } S_h.
\end{eqnarray}
By the definition of $\varphi_h$, we get
\begin{lemma} \label{lem_phih1}
The map $(h,x) \mapsto \varphi_h(x)$ is smooth on all domains:
$$
\{(h,x), h \in (0,h_M), x \in F_h\} \,,  \, \{(h,x) \in (0,h_M), x \in S_h\}\,,\,  \{(h,x), h \in [0,h_M), x \in F_h \setminus \Omega_{h,\delta}\}.
$$
Furthermore, there holds:
\begin{equation} \label{eq_contOmega}
\varphi_h\cdot n=0 \quad \text{ on $\pa \Omega$}\\
\end{equation}
and
\begin{equation} \label{eq_contSh}
\varphi_h|_{F_h} \cdot n = e_3 \cdot n \quad \text{ on $\pa S_h$}.
\end{equation}
\end{lemma}
This comes from direct calculations, that we leave to the reader for brevity.
We now state some refined estimates,  to be used  in the proof of {Theorem \ref{thm_slip}}. The whole point is to have an accurate bound on the blow-up of the test function as $h$ goes to zero.
\begin{lemma}\label{est:1}
There exist constants $0<c<C$ so that, when $h< h_M$ :
\begin{eqnarray}
\|\varphi_h\|_{L^2(F_h)} &\leq& C,\label{est:L2} \\
\|\nabla \varphi_h\|^2_{L^2(F_h)}&\leq& C|\ln{h}|,\label{est:gradL2_1}\\
\|D(\varphi_h)\|^2_{L^2(F_h)}  &\geq&  c |\ln(h)|, \label{est:gradL2_2}
\end{eqnarray}
and
\begin{equation} \label{est:bdy}
\|2\beta_S (D({\varphi}_h) n)\times n+({\varphi}_h-e_3)\times n\|_{L^2(\pa S_h)} \leq C,
\end{equation}
Furthermore, for $i=1,2,3,$ denoting by exponent $i$ cartesian coordinates, there holds:
\begin{eqnarray}
\|\int^{h+\gamma_s(r)}_0\pa_h\varphi^i_hdz\|^2_{L^2(\partial\Omega \cap \partial \Omega_{h,\delta} )}&\leq& C|\ln{h}|, \label{dhpsi1}\\
 \|\int^{h+\gamma_s(r)}_z\pa_h\varphi^i_hds\|_{L^2(\Omega_{h,\delta})}&\leq& C,\label{dhpsi2}\\
\|(h+\gamma_s(r))\pa_i\int^{h+\gamma_s(r)}_z\pa_h\varphi_h^i(x_1,x_2,s)ds\|_{L^2(\Omega_{h,\delta})}&\leq& C, \label{dhpsi3}
\end{eqnarray}
and for $(i,j) \in \{1,2,3\}^2$ and $x \in \Omega_{h,\delta}$:
\begin{eqnarray}
\left|\int^{h+\gamma_S(r)}_z(\pa_i\varphi_h^j(x_1,x_2,s)+\pa_j\varphi_h^i(x_1,x_2,s))ds\right|&\leq& C, \label{pctlP_zvarphi}\\
\left|(h+\gamma_S(r))\pa_i\int^{h+\gamma_S(r)}_z(\pa_i\varphi_h^j(x_1,x_2,s)+\pa_j\varphi_h^i(x_1,x_2,s))ds\right|&\leq& C.
\end{eqnarray}
\end{lemma}
\noindent The proof of this lemma is postponed to Appendix \ref{app}.

\medskip

When computing \eqref{drag}, we introduced a pressure $q.$ The aim was to smoothen
singularities in $\Delta \varphi_h$ which must arise when $h \to 0.$
The construction of this pressure is the content of the next lemma:
\begin{lemma}\label{est:2}
There exists a map $(h,x) \mapsto q_h$ which is smooth on the domains
$$
\{(h,x), h \in (0,h_M), x\in F_h  \} \text{ and } \{(h,x), h \in [0,h_M), x\in F_h  \setminus \Omega_{h,\delta} \}
$$
and which satisfies :
\begin{itemize}
\item there exist constants $0 < c < C < \infty$ such that :
\beno
 c|\ln{h}|\leq \int_{\partial S_h}(2D(\varphi_h)-q_hI)n\cdot (e_3-\varphi_h)d\sigma-\int_{\partial \Omega}(2D(\varphi_h)-q_hI)n\cdot \varphi_hd\sigma\leq C|\ln{h}| ;   \\
\eeno
\item there exists a positive constant $C$ such that, for any $h \in (0,h_M)$ and $v \in H^1(F_h)$
satisfying $v\cdot n =0$ on $\partial \Omega$, we have:
\begin{equation} \label{eq_deltaphi}
\left| \int_{F_h}(\Delta \varphi_h-\nabla q_h)v\right|
\leq C \left( \|D(v)\|_{L^2(F_h)} +\|v\|_{L^2(\pa \Omega)}\right).
\end{equation}
\end{itemize}
\end{lemma}
\noindent The proof of this lemma is also postponed to Appendix \ref{app}.

\subsection{Conclusion of the proof of Theorem \ref{thm_slip}} \label{sec_prfslip}
We are now in a position to prove Theorem \ref{thm_slip}. Combining the energy inequality \eqref{nrj_est_s} with the classical identity
$$ - \int_0^t \int_{F(t)} \rho_F g e_z \cdot u_F - \int_0^t \int_{S(t)} \rho_S g e_z \cdot u_S  = \dfrac{4(\rho_F - \rho_S)g\pi}{3} (h(t)-h(0)) $$
we obtain easily  (remember that $\rho_S > \rho_F$):
\begin{multline}
\sup_{(0,T)} \int_{\Omega} |u(t)|^2  + \int_{0}^T \int_{F(t)} |D(u_F(t))|^2 \\
+\int_0^T\int_{\partial \Omega} |u_F \times n|^2 {\rm d}\sigma   +  \int_0^T \int_{\partial S(t)} |(u_F - u_S) \times n |^2 {\rm d}\sigma
\leq C_0. \label{eq_nrj} \end{multline}
This energy bound will be of constant use.

\medskip
As mentioned before, we  take a special test function in the variational formulation \eqref{ws_s}. Namely, we take $\varphi(t,x) =  \zeta(t) \, \varphi_{h(t)}(x)$, with the field  $\varphi_h$ built in paragraph \ref{sec_tf}. From Lemma \ref{lem_phih1}, we have that $\varphi \in \mathcal T_T$ so that identity \eqref{drag} holds true.
The whole point is then to bound properly the r.h.s. of  (\ref{drag}):
\begin{eqnarray*}
RHS &\triangleq&\int_0^T\int_{F(t)}\rho_F u_F\cdot\pa_t\varphi_F  +   \int_0^T\int_{S(t)}\rho_S u_S\cdot\pa_t\varphi_S +\int_0^T\int_{F(t)}\rho_Fu_F\otimes u_F:\nabla \varphi_F\nonumber\\
& -&\int_0^T\left(\int_{F(t)}\rho_F ge_z\varphi_F  + \int_{S(t)}\rho_S g e_z \varphi_S \right)+\left( \int_{F_0}\rho_F u_{F,0}\varphi_F(0,\cdot) + \int_{S_0}\rho_S u_{S,0}\varphi_S(0,\cdot)\right)\nonumber\\
&-&\mu_F  \int_0^T \int_{\pa \Omega}(u_F\times n)\cdot  \left( \f{1}{\beta_\Omega}(\varphi_F\times n) +  (2 D(\varphi_F)-qI)n \right) d\sigma\\
&-&\mu_F \int_0^T  \int_{\pa S(t)}((u_F-u_S)\times n) \cdot \left(  \f{1}{\beta_S}  (\varphi_F-\varphi_S)\times n)  + (2 D(\varphi_F)-qI)n \right) d\sigma   \nonumber\\
&+& \mu_F \int_0^T\int_{F(t)}(\Delta \varphi_F-\nabla q)\cdot u_F \nonumber\\
&=&I_1 + I_2 + I_3 + I_4 + I_5 - \mu_F ( I_6 + I_7 + I_8),
\end{eqnarray*}
where we have applied the continuity of normal traces on $\partial S(t)$ and $\partial \Omega$ in the  boundary terms.
\begin{itemize}
\item
We first deal with
\end{itemize}
$$
I_3= \int_0^T\int_{F(t)}\rho_Fu_F\otimes u_F: \nabla \varphi_F \triangleq  \int_0^T \rho_F\zeta(t)I_3(t).  \
$$
We introduce the functions $x \mapsto P_z \varphi^{i,j}_h(x)$ for $(i,j) \in \{1,2,3\}^2$ defined by:
\beno
 P_z \varphi^{i,j}_h(x)& =&  \int^{h+\gamma_S(r)}_{z} ( \pa_i {\varphi}_h^j+\pa_j{\varphi}_h^i)(x_1,x_2,s) ds   \,, \quad \forall\,  r < 1, \quad z < h + \gamma_s(r),\\
P_z \varphi^{i,j}_h(x) &=& \int^{h+1}_{z} ( \pa_i {\varphi}_h^j+\pa_j{\varphi}_h^i) (x_1,x_2,s) ds \quad \mbox{ otherwise}.
\eeno
Let us stress that $P_z \varphi^{i,j}_h$ is smooth in $F_h$,  notably across $r=1$: it follows from the cancellation of  the symmetric gradient $\pa_i {\varphi}_h^j +\pa_j{\varphi}_h^i$ near  $(r=1,z=1+h)$ that both expressions coincide there.

\medskip
Moreover,  by \eqref{pctlP_zvarphi}, there exists a constant $C$ independent of $t$ such that, for all $(i,j) \in \{1,2,3\}^2,$ there holds:
\begin{eqnarray}
|P_z \varphi_h^{i,j}(x_1,x_2,x_3)| +  |(h+\gamma_S(r)) \partial_{i} P_z \varphi_h^{i,j}(x_1,x_2,x_3)| \leq C, \quad \forall \, x \in \Omega_{h(t),\delta},  \label{Pzphi1}\\
|P_z \varphi_h^{i,j}(x_1,x_2,x_3)| +  | \partial_{i} P_z \varphi_h^{i,j}(x_1,x_2,x_3)| \leq C, \quad \forall \, x \in F(t) \setminus \Omega_{h(t),\delta}.\label{Pzphi2}
\end{eqnarray}
Then, by performing integration by parts, we obtain (summation on indices $i$ and $j$ is implicit):
\beno
I_3(t)&=&  - \int_{F(t)}u_F^iu_F^j \, \pa_z[P_z \varphi^{i,j}_{h}] \: =  \: \int_{F(t)}\pa_3(u_F^iu_F^j) P_z \varphi^{i,j}_{h} -\int_{\pa \Omega} u_F^iu_F^j P_z \varphi^{i,j}_h n^3{\rm d}\sigma \\
&= & \int_{F(t)}((\pa_3u_F^i+\pa_iu_F^3)u_F^j+(\pa_3u_F^j+\pa_ju_F^3)u_F^i) P_z \varphi^{i,j}_h \\
& & \quad -  \int_{F(t)}(\pa_i u_F^3 u_F^j+\pa_j u_F^3 u_F^i) P_z\varphi^{i,j}_h  -\int_{\pa \Omega} (u_F^iu_F^j)P_z \varphi^{i,j}_hn^3{\rm d}\sigma \\
&= &
\int_{F(t) }((\pa_3u_F^i+\pa_iu_F^3)u_F^j+(\pa_3u_F^j+\pa_ju_F^3)u_F^i) P_z \varphi^{i,j}_h\\
&&\quad + \int_{F(t)} \left( (\pa_iu_F^j+\pa_ju_F^i)u_F^3 P_z \varphi^{i,j}_h  + u_F^3u_F^j\pa_i P_z \varphi^{i,j}_h   + u_F^3 u_F^i\pa_j P_z \varphi^{i,j}_h \right) \\
&& \quad -\int_{\pa \Omega} (u_F^iu_F^j n^3 + u_F^ju_F^3 n^i+ u_F^i u_F^3 n^j )P_z \varphi^{i,j}_h  {\rm d}\sigma.\\
\eeno
In this last identity, we apply \eqref{Pzphi1}-\eqref{Pzphi2} to obtain :
$$
\left|\int_{\pa \Omega} (u_F^iu_F^j n^3 + u_F^ju_F^3 n^i+ u_F^i u_F^3 n^j )P_z \varphi^{i,j}_h  {\rm d}\sigma \right| \leq C   \|u_F(t)\|^2_{L^2(\pa \Omega)}\,.
$$
Concerning the other terms, we have for instance :
\beno
\left| \int_{F(t) }u_F^3u_F^j\pa_i P_z \varphi^{i,j}_h \right| &\leq&
\left| \int_{F(t) \cap \Omega_{h(t),\delta} }\dfrac{u_F^3}{h+\gamma_S(r)} u_F^j(h+\gamma_S(r))\pa_i P_z \varphi^{i,j}_h \right|\\
&&\quad +
\left| \int_{F(t) \setminus \Omega_{h(t),\delta} } u_F^3u_F^j\pa_i P_z \varphi^{i,j}_h \right|
\eeno
\vspace{-18pt}
\beno
&\leq &
\|\f{u_F^3}{h+\gamma_S(r)}\|_{L^2(F(t) \cap \Omega_{h(t),\delta})} \|u_F\|_{L^2(F(t))} \|(h+\gamma_S(r))\pa_i P_z \varphi^{i,j}_h\|_{L^{\infty}(F(t) \cap \Omega_{h(t),\delta}))}
\\
&&\quad  +
\|{u_F^3}\|_{L^2(F(t) \cap \{|(x_1,x_2)| < (1+d_\delta)\})} \|u_F\|_{L^2(F(t))} \|\pa_i P_z \varphi^{i,j}_h\|_{L^{\infty}(F(t) \setminus \Omega_{h(t),\delta}))}.
\eeno
As $u^3_F(x) = 0$ for $x \in \{ |(x_1,x_2)| < 1+ d_\delta, x_3 =0\},$ a straightforward Poincar\'e inequality yields that
there exists a constant $C >0$ independent of time so that:
\beno
\|\f{u_F^3}{h+\gamma_S(r)}\|_{L^2(F(t) \cap \Omega_{h(t),\delta})} &\leq&  C \|\pa_3 u_F^3\|_{L^2(F(t))} \\
\|{u_F^3}\|_{L^2(F(t) \cap \{ |(x_1,x_2)| < (1+d_\delta) \})} &\leq & C \left(   \|\pa_3 u_F^3\|_{L^2(F(t))} +  \|u_F \times n \|_{L^2(\pa \Omega)}\right).
\eeno
Introducing \eqref{Pzphi1}--\eqref{Pzphi2} then yields :
$$
|I_3(t)| \leq C \left[ \|D(u_F)\|^2_{L^2(F(t))} + \|D(u_F(t))\|_{L^2(F(t))} \|u_F(t)\|_{L^2(F(t))}  + \|u_F(t)\|^2_{L^2(\pa \Omega)}\right],
$$
so that  \eqref{eq_nrj} implies there exists a constant $C$ depending only on initial data for which
$$
|I_3| \leq C(1+\sqrt{T}) \|\zeta\|_{L^{\infty}(0,T)}\,.
$$
\begin{itemize}
\item  To compute $I_1$, we split:
\end{itemize}
\begin{equation*}
I_1  =  \int_0^T\int_{F(t)}\rho_F u_F\cdot\pa_t\varphi_F = \int_0^T\int_{F(t)}\rho_F u_F \zeta'(t) \varphi_{h(t)} + \int_0^{T} \int_{F(t)}\rho_F u_F \zeta(t) h'(t)\pa_h\varphi_F \, \triangleq I_1^a + I_1^b.
\end{equation*}
Applying \eqref{est:L2} yields :
$$
|I_1^a| \leq C \|\zeta'\|_{L^1(0,T)} \|u\|_{L^{\infty}(0,T;L^2(\Omega))} \leq C_0  \|\zeta'\|_{L^1(0,T)}\,.
$$
Then, we have $I_1^b = \rho_F \int_0^t \zeta(t) h'(t) I_1^b(t)$ where, for all $t \in (0,T)$:
$$
I_1^b(t)= \sum_{i=1}^3 \int_{F(t)}u_F^i\cdot\pa_h\varphi_{h(t)}^i.
$$
To bound this term, we define as previously $P_z \varphi_h^i$ which satisfies, for all $i=1,2,3$ :
\beno
 P_z \varphi^i_h(x)& =&  \int^{h+\gamma_S(r)}_{z} \pa_h{\varphi}_h^i(x_1,x_2,s) ds \,, \quad \forall\,  r < 1, \quad z < h + \gamma_s(r),\\
P_z \varphi^i_h(x) &=& \int^{h+1}_{z}\pa_h{\varphi}_h^i(x_1,x_2,s) ds  \quad \mbox{ otherwise}.
\eeno
Again, $P_z \varphi^i_h$ is continuous in $F_h$, notably near $r=1$, as $\pa_h \varphi_h$ vanishes near
$(r=1, z=1+h)$.

With \eqref{dhpsi2} together with \eqref{dhpsi3}, there is  a constant $C$ such that:
\begin{equation} \label{eq_Pzphi1}
\| P_z \varphi_h^{i} \|_{L^2(F(t))} +  \| (h+\gamma_S(r)) \partial_{i} P_z \varphi_h^{i}\|_{L^2(F(t)\cap \Omega_{h(t),2\delta})} + \|\partial_{i} P_z \varphi_h^{i}\|_{L^2(F(t)\setminus \Omega_{h(t),2\delta})}  \leq C,
\end{equation}
\begin{equation} \label{eq_Pzphi2}
\|P_z \varphi_h^i\|_{L^2(\pa \Omega)} \leq C |\ln(h)|^{1/2}.
\end{equation}
From now on, we restrict implicitly  the integral $I_1^b(t)$ to $|(x_1,x_2)| < (1+ d_\delta)$ as $\partial_h \varphi_h$ vanishes elsewhere.
Performing integration by parts (summation over index $i$ is implicit), we obtain :
\beno
I_1^b(t)&=&- \int_{F(t)}  u_F^i \pa_zP_z \varphi_h^i\\
& =  & \int_{F(t)} \pa_3 u_F^i P_z \varphi_h^i -  \int_{\pa \Omega} u_F^i P_z \varphi_h^i n^3  {\rm d}\sigma  \\
&\leq & \int_{F(t)} (\pa_3 u_F^i + \pa_i u_F^3 ) P_z \varphi_h^i  +   \int_{F(t)} u_F^3 \partial_ i P_z \varphi_h^i  - \int_{\pa \Omega} (u_F^3 n^i + u_F^i n^3) P_z \varphi_h^i
\eeno
Relying on \eqref{eq_Pzphi1}-\eqref{eq_Pzphi2} and \eqref{eq_nrj}, we might then reproduce the arguments we used to compute $I_3,$ resulting in:
\beno
 |I_1^b| \leq  C \|\zeta\|_{L^{\infty}} \left( \int^T_0  |h'(t)| |\ln{h}|^{\frac{1}{2}}\|u_F \times n\|_{L^2(\partial \Omega)}+  \sqrt{T}\right).
\eeno
At this point, we remark that, given the computations of \cite[Section 3.2]{DGVH12},  $|\ln(h(t))|^{1/2}$ is the minimum dissipation energy
of functions $v \in H^1(F(t))$ satisfying $v \cdot n = e_3 \cdot n$ on $\partial S(t)$ and $v \cdot n = 0$ on $\pa \Omega.$ Consequently, there exists a constant
$C$ independent of time such that:
$$
|\dot{h}| |\ln(h)|^{\frac{1}{2}} \leq C \left(  \int_{F(t)} |D(u_F(t))|^2 + \int_{\partial \Omega} |u_F \times n|^2 {\rm d}\sigma   +  \int_{\partial S(t)} |(u_F - u_S) \times n |^2 {\rm d}\sigma   \right)^{\frac{1}{2}},
$$
and \eqref{eq_nrj} implies:
$$
\int^T_0  |h'(t)| |\ln{h}|^{\frac{1}{2}}\|u_F \times n\|_{L^2(\partial \Omega)} \leq C\,.
$$
\begin{itemize}
\item  Next, we recall that
$\int_{F(t)}\rho_F ge_3\varphi_{h(t)} = - \dfrac{4 \pi \rho_F g}{3}$,
which implies:
\end{itemize}
\ben
I_4 = -\int^T_0\int_{F(t)}\rho_F ge_3\varphi_F-\int^T_0\int_{S(t)}\rho_S g e_3\varphi_S=  \dfrac{4 \pi (\rho_F-\rho_S) g}{3} \int_0^T \zeta(t){\rm d}t.
\een
\begin{itemize}
\item Then, we apply \eqref{bdy1} together with  \eqref{eq_nrj}. It leads to
\end{itemize}
\beno
|I_6|&=& \left|\int_0^T \zeta(t) \int_{\pa \Omega}(u_F\times n)\cdot  \left( \f{1}{\beta_\Omega}(\varphi_h\times n) +  (2 D(\varphi_h)-qI)n \right) d\sigma\right|, \\
      &=&\left| \int_0^T \zeta(t) \int_{\pa \Omega}(u_F\times n)\cdot  \left( \f{1}{\beta_\Omega}(\varphi_h\times n) +  2 D(\varphi_h) n \right) d\sigma  \right|,\\
      &\leq & C  \int_0^T |\zeta(t)| \|u_F\times n\|_{L^2(\pa \Omega)} \leq C'  \|\zeta\|_{L^{\infty}(0,T)} \sqrt{T},
 \eeno
\begin{itemize}
\item By \eqref{est:bdy} together with \eqref{eq_nrj}:
\end{itemize}
\beno
|I_7| &=& \left| \int_0^T  \int_{\pa S(t)}((u_F-u_S)\times n) \cdot \left(  \f{1}{\beta_S}  (\varphi_h-e_3)\times n)  + 2 D(\varphi_h) n \right) d\sigma\right| \\
       &\leq& C\int_0^T |\zeta(t)| \|(u_F-u_S)\times n\|_{L^2(\pa S(t))} \|(\varphi_h - e_3)\times n) + 2 \beta_S D(\varphi_h) n\|_{L^2(\pa S(t))}\text{d$t$}\\
  &\leq& C' \|\zeta\|_{L^{\infty}(0,T)} \sqrt{T}.
\eeno
\begin{itemize}
\item As for $I_8$, applying \eqref{eq_deltaphi}, we have that
\end{itemize}
\beno
I_8 = \int^T_0\int_{F(t)}(\Delta \varphi_F-\nabla q)u_F&\leq& C\int^T_0 |\zeta(t)|\left( \|D(u_F)\|_{L^2(F(t))}+\|u_F\|_{L^2(\pa \Omega)}\right)dt\\
					&\leq& C' \|\zeta\|_{L^{\infty}(0,T)} \sqrt{T}.
\eeno
\begin{itemize}
\item  Eventually, we get easily the following controls:
\end{itemize}
$$ | I_2 |  \: \le \:  C \, \| \zeta' \|_{L^1(0,T)}, \quad   I_5 \le C.$$

Combining all these estimates to bound $RHS$, we end up with:
\begin{multline} \label{finaldrag}
\mu_F \int_0^T  h'(t) \zeta(t) n(h(t))\text{d$t$} \\
= \dfrac{4 \pi (\rho_F-\rho_S) g}{3} \int_0^T \zeta(t){\rm d}t + O (C_0 (\|\zeta\|_{L^{\infty}(0,T)}  + \|\zeta'\|_{L^1(0,T)}) (1+\sqrt{T}) )
\end{multline}

\medskip
We conclude by introducing a sequence  $\zeta_n$ in ${\cal D}([0,T))$ such that
$$ \zeta_n \rightarrow 1_{[0,T)} \: \mbox{ pointwise,}  \quad \mbox{with} \quad  \| \zeta_n \|_{L^\infty} \le C, \quad \| \zeta'_n \|_{L^1(0,T)} \le C. $$
A possible choice is $\zeta_n(t) = 1_{[0,T)}(t) - \chi(n(T-t))$, for some smooth $\chi$ compactly supported with $\chi = 1$ near $0$. We take $\zeta = \zeta_n$ in \eqref{finaldrag}, and let $n$ go to infinity, to get
\begin{equation} \label{eq_contradict}
\mu_F \int_{h(0)}^{h(T)} n(s)\text{d$s$} +  \dfrac{4 \pi (\rho_S - \rho_F) g}{3} T  \leq  C_0 (1+\sqrt{T}).
\end{equation}
Introducing the bounds of Lemma \ref{est:1} and Lemma \ref{est:2} into \eqref{n(h)}, we show there exists a constant $C$ such that:
\beno
n(h)\leq C |\ln{h}| \quad \forall \, h < h_M,
\eeno
which implies that  there exists a constant $C$ for which:
\beno
\int^{h}_{h(0)} n(s) ds \geq -C\,, \quad \forall \, h < h_M\,.
\eeno
Hence, we obtain  a contradiction in \eqref{eq_contradict} if $T$ is large enough compared to $C$ and $C_0$ which
are fixed by $h_M$ and the initial data only.

\section{Proof of Theorem \ref{thm_mixed}} \label{sec_mix}
The proof is very close to the one performed above for slip boundary conditions.
Therefore, we only sketch the proof, insisting on  its few specific features.

\paragraph{Computation of the drag.}
Let $(S,u)$ be a weak solution given by Theorem \ref{thm:ws_m} on $[0,T)$. As before,  we consider a solution up to its "maximal time of existence": that is $T=\infty$ if there is no collision,  or $T < \infty$ the time of collision.  We assume that $(S,u)$  satisfies  A1-A4 over $[0,T)$. This time, we want to show that there is no collision, which means that $h$ does not go to zero in finite time (and thus $T=\infty$).  We shall consider a test function $\varphi \in \mathcal{T}_T^m$ of the form
\begin{equation} \label{eq_varphi_m}
\varphi  = \zeta(t) \varphi_{h(t)} \qquad \text{ with } \qquad  \varphi_S = \zeta(t)e_z\,.
\end{equation}
associated with a pressure $q \in \mathcal D([0,T) \times \bar{\Omega}),$ of the form
$$
q = \zeta(t) q_{h(t)}.
$$
The values of $\varphi_h$ and $q_h$ are given  below. As in the previous section,  we shall use functions $\zeta = \zeta_n$ approximating $1_{(0,T)}.$
Here and in what follows, we keep the convention that $\varphi_{F} = \varphi 1_{F(t)}$ and $\varphi_{S} = \varphi 1_{S(t)}\,.$
With  computations similar to those of the previous section, we obtain:
\ben\label{drag-m}
&&\mu_F \int_0^T \zeta(t) h'(t)n(h(t)) \text{d$t$}\nonumber\\
&=&\int_0^T\int_{F(t)}\rho_F u_F\cdot\pa_t\varphi_F  + \int_0^T\int_{F(t)}\rho_Fu_F\otimes u_F:\nabla \varphi_F\nonumber\\
& -&\int_0^T\int_{F(t)}\rho_F ge_z\varphi_F -\int_0^T\int_{S(t)}\rho_S g e_z \varphi_S + \int_{F_0}\rho_F u_{F,0}\varphi_F(0,\cdot) + \int_{S_0}\rho_S u_{S,0}\varphi_S(0,\cdot)\nonumber\\
&-&\f{\mu_F}{\beta_\Omega}\int_0^T\int_{P}(u_F\times n)\cdot(\varphi_F\times n) \nonumber\\
&+& \mu_F \Bigg( \int_0^T\int_{F(t)}(\Delta \varphi_F-\nabla q)\cdot u_F - \int_0^T\int_{\partial \Omega}(2 D(\varphi_F)-qI)n\cdot u_F d\sigma \Bigg)
\een
Again, we compute $n(h)$ applying the identity:
\begin{eqnarray}
n(h)
&=& \int_{F_h}(\Delta \varphi_{h}-\nabla q_{h}) \cdot \varphi_{h} dx+ 2 \int_{F_h}D(\varphi_{h}) : D(\varphi_{h}) dx \notag \\
& - & \int_{\pa \Omega}( 2D(\varphi_{h})-q_{h}I)n\cdot \varphi_{h} d\sigma. \label{n(h)-m}
\end{eqnarray}

\paragraph{Construction of the test-function.}
For the function $(x,h) \mapsto \varphi_h(x),$ we keep the structure of \eqref{test_s} with a test-function $\widetilde{\varphi}_h$
still of the form \eqref{tildephih}. The only point that differs is the choice of the coefficients in the  polynomial $\Phi.$ We set now:
$$
P_1(r) \triangleq \dfrac{6}{4+\alpha_P} \quad P_2(r) \triangleq  \dfrac{3\alpha_P}{4+ \alpha_P} \qquad P_3(r) \triangleq - 2 \dfrac{1+ \alpha_P}{4+ \alpha_P}\,,
$$
with $\alpha_P$ given by \eqref{def_Phifin}. We remark that this is equivalent to taking the limit $\alpha_S \to \infty$ in our choice of test-function in the previous
section.   With this choice, the function $\Phi$ satisfies boundary conditions:
\begin{itemize}
\item for $z=0$
\begin{eqnarray}
&&\Phi(r,0) = 0, \\
&& \partial_{zz} \Phi\left(r,\f{z}{h+\gamma_s(r)}\right) - \dfrac{1}{\beta_{\Omega}} \partial_{z} \Phi\left(r,\f{z}{h+\gamma_s(r)}\right) = 0
\end{eqnarray}
\item for $z= h + \gamma_S(r)$
\end{itemize}
\begin{eqnarray}
&&  \Phi(r,1) = 1\, \\
&&\partial_z \Phi(r,1) = 0.
\end{eqnarray}
Similarly to the previous case, we provide also the following technical lemmas :
\begin{lemma} \label{lem_phih2}
The map $(h,x) \mapsto \varphi_h(x)$ is smooth on all domains :
$$
\{(h,x), h \in (0,h_M), x \in F_h\} \,,  \, \{(h,x) \in (0,h_M), x \in S_h\}\,,\,  \{(h,x), h \in [0,h_M), x \in F_h \setminus \Omega_{h,\delta}\}.
$$
Furthermore, there holds:
\begin{equation} \label{eq_contOmega_2}
\varphi_h\cdot n=0 \quad \text{ on $\pa \Omega$}\\
\end{equation}
and
\begin{equation} \label{eq_contSh_2}
\varphi_h|_{F_h}= e_3  \quad \text{ on $\pa S_h$}.
\end{equation}
\end{lemma}
\begin{lemma}\label{est_m:1}
There exist constants $0<c<C$ so that, when $h< h_M$ :
\begin{eqnarray}
\|\varphi_h\|_{L^2(F_h)} &\leq& C,\label{est_m:L2} \\
\|\nabla \varphi_h\|^2_{L^2(F_h)}&\leq& \dfrac{C}{h},\label{est_m:gradL2_1}\\
\|D(\varphi_h)\|^2_{L^2(F_h)}  &\geq&  \dfrac{c}{h}, \label{est_m:gradL2_2}
\end{eqnarray}
Furthermore, for $i=1,2,3,$ denoting by exponent $i$ cartesian coordinates, there holds:
\begin{eqnarray}
\|\int^{h+\gamma_s(r)}_0\pa_h\varphi^i_hdz\|_{L^2(\partial\Omega \cap \partial \Omega_{h,\delta} )}&\leq& C|\ln{h}|, \label{dhpsi1_2}\\
 \|\int^{h+\gamma_s(r)}_z\pa_h\varphi^i_hds\|_{L^2(\Omega_{h,\delta})}&\leq& C,\label{dhpsi2_2}\\
\|(h+\gamma_s(r))\pa_i\int^{h+\gamma_s(r)}_z\pa_h\varphi_F^i(x_1,x_2,s)ds\|_{L^2(\Omega_{h,\delta})}&\leq& C, \label{dhpsi3_2}
\end{eqnarray}
and
\begin{itemize}
\item for $i \in \{1,2,3\}^2,$ and $x \in \Omega_{h,\delta}$:
\begin{eqnarray} \label{d3phiplus}
\Big| \pa_i\varphi_h^3 +\pa_3\varphi_h^i\Big|&\leq& \dfrac{Cr}{(h+ \gamma_S(r))^2}, \label{pctlP_zvarphi2}
\end{eqnarray}
\item for $(i,j) \in \{1,2\}^2,$ and $x \in \Omega_{h,\delta}$:
\begin{eqnarray}
\left|\int^{h+\gamma_S(r)}_z(\pa_i\varphi_h^j(x_1,x_2,s)+\pa_j\varphi_h^i(x_1,x_2,s))ds\right|&\leq& C, \label{pctlP_zvarphi2bis}\\
\left|(h+\gamma_S(r))\pa_i\int^{h+\gamma_S(r)}_z(\pa_i\varphi_h^j(x_1,x_2,s)+\pa_j\varphi_h^i(x_1,x_2,s))ds\right|&\leq& C.
\end{eqnarray}
\end{itemize}
\end{lemma}
\begin{lemma}\label{est_m:2}
There exists a map $(h,x) \mapsto q_h$ which is smooth on the domain
$
\{(h,x), h \in (0,h_M), x\in F_h  \}
$
and which satisfies :
\begin{itemize}
\item there exists a constant positive constant $C$ such that :
\beno
\left|  \int_{\partial \Omega}(2 D(\varphi_h)-q_hI)n\cdot \varphi_hd\sigma \right| \: \leq \: C|\ln{h}|  ;   \\
\eeno
\item there exists a positive constant $C$ such that, for any $h \in (0,h_M)$ and $v \in H^1(\Omega)$
satisfying $v_{|_{S_h}} \in \mathcal R$ and $v\cdot n =0$ on $\partial \Omega$, we have:
\begin{equation} \label{eq_deltaphi2}
\left| \int_{F_h}(\Delta \varphi_h-\nabla q_h)v\right|
\leq C \left( \|D(v)\|_{L^2(F_h)} +\|v\|_{L^2(\Omega)} \right).
\end{equation}
\end{itemize}
\end{lemma}
More details concerning these lemmas  are provided in Appendix  \ref{app_2}.

\paragraph{Conclusion.} Introducing the results of Lemma \ref{est_m:1} and Lemma \ref{est_m:2},
we compute again the right-hand side of \eqref{drag-m} with the help of energy estimates satisfied by $u.$ We only point out
two differences. First, there exists a universal constant $C$ so that
$$
 \|u_F\|_{L^2(\partial S(t))} = \| u_S \|_{L^2(\pa S(t))}  \leq C \|u_S\|_{L^2(S(t))} \le C \|u\|_{L^2(\Omega)},
$$
by taking into account  that $u_S$  is a rigid vector field. Second, one term requires a new treatment (we drop summation over index $i \in \{1,2,3\} $ for simplicity):
\beno
& &\int_{F(t)}u_F^iu_F^3(\pa_i\varphi_F^3+\pa_3\varphi_F^i)\\
&=&\int_{F(t) \cap \Omega_{h(t),2\delta}} \chi \f{u^3_F}{\gamma_s(r)+h}u^i_F(\gamma_s(r)+h)(\pa_i\varphi_F^3+\pa_3\varphi_F^i) + O(\|u_F\|^2_{L^2})\\
&\leq&\|u^i_F(\gamma_s(r)+h)(\pa_i\varphi_F^3+\pa_3\varphi_F^i)\|_{L^2(F(t)\cap \Omega_{h(t),2\delta})}\|\f{u^3_F}{\gamma_s(r)+h}\|_{L^2(F(t) \cap \Omega_{h(t),2\delta})} + O(\|u_F\|^2_{L^2})\\
&\leq& C\|D(u_F)\|_{L^2(F(t))}\| \chi u^i_F(\gamma_s(r)+h)(\pa_i\varphi_F^3+\pa_3\varphi_F^i)\|_{L^2(F(t) \cap \Omega_{h(t),2\delta})} + O(\|u_F\|^2_{L^2(F(t))}).
\eeno
where $\chi$ is given by \eqref{eq_chi}. Because of \eqref{d3phiplus}, there holds
\beno
&&\|\chi u^i_F(\gamma_s(r)+h)(\pa_i\varphi_F^3+\pa_3\varphi_F^i)\|^2_{L^2(F(t) \cap \Omega_{h,2\delta})}\\
&\leq& \int_{F(t)}|u^i_F|^2\cdot \f{\chi}{\gamma_s(r)+h} \\
&=& \int_{F(t)} (\pa_zu^i_F+\pa_i u^3_F) \chi u^i_F \f{z-\gamma_s(r)-h}{\gamma_s(r)+h}-\int_{F(t)} \chi u^i_F \pa_i u^3_F  \f{z-\gamma_s(r)-h}{\gamma_s(r)+h} \\
&& \quad + O(\|u_F\|^2_{L^2(F(t))}+  \|u_F\|^2_{L^2(\pa \Omega)})\\
&=& \int_{F(t)}\dfrac{u^3_F}{\gamma_s(r)+h} u^i_F \chi (\gamma_s(r)+h) \pa_i \left( \f{z-\gamma_s(r)-h}{\gamma_s(r)+h}\right)\\
 && \quad + O\Big(\|u_F\|^2_{L^2(F(t))}+  \|u_F\|^2_{L^2(\pa \Omega)} + \|D(u_F)\|_{L^2(F(t))}\|u_F\|_{L^2(F(t))}\Big)\\
&=&O\Big(\|u_F\|^2_{L^2(F(t))}+  \|u_F\|^2_{L^2(\pa \Omega)} + \|D(u_F)\|_{L^2(F(t))}\|u_F\|_{L^2(F(t))}\Big).
\eeno
which implies that
\beno
\left| \int_{F(t)}u_F^iu_F^3(\pa_i\varphi_F^3+\pa_3\varphi_F^i) \right|&=& O\Bigg( \|D(u_F)\|_{L^2(F(t))}(\|D(u_F)\| +\|u_F\|_{L^2(\pa \Omega)})
											+ \|u_F\|^2_{L^2(F(t))}\Bigg)
\eeno
Combining with estimates similar to the previous section for the other terms, we arrive at:
\begin{equation} \label{eq_contradict2}
\left| \mu_F \int_{h(0)}^{h(T)} n(s)\text{d$s$}\right|  \leq  C_0 (1+T).
\end{equation}
Introducing the bounds of Lemma \ref{est_m:1} and Lemma \ref{est_m:2} into \eqref{n(h)}, we show there exists a constant $C$ such that:
\beno
n(h)\geq \dfrac{C}{h} \,, \quad \forall\, h \in (0, h_M),
\eeno
so that \eqref{eq_contradict2} yields finally :
\beno
|\ln(h(T))| \leq C_0(1+ T)\,, \quad \forall \, T >0.
\eeno
This ends the proof.

\paragraph{Acknowledgements.} 
The last author is supported by the Project Instabilit\'es hydrodynamiques funded partially by Mairie de Paris (program ÒEmergence(s)Ó)
and Fondation Sciences Math\'ematiques de Paris.

\appendix

\section{Estimates on the test-function in the slip case}\label{app}
In this section, we collect estimates on the test-function of Section \ref{sec_slip}
as defined in \eqref{test_s}. For further computations, we define,
$$
\Psi(r,z) := \Phi\left(r, \dfrac{z}{h+\gamma_S(r)}\right)\,, \quad \forall  \, h \in (0,h_M),  \quad \forall \,(r,z) \in \Omega_{h,1}\,.
$$
In particular, we note that $\Psi(r,h+\gamma_S(r)) = 1$ so that:
\begin{eqnarray*}
\dfrac{\text{d}}{\text{d$r$}} \left[ \dfrac{r^2}{2}\Psi(r,h+\gamma_S(r))\right] &=& r \, \\
					&=& \partial_{r}\left( \dfrac{r^2}{2} \Psi\right)(r,h+\gamma_S(r))  + \dfrac{r}{\sqrt{1-r^2}} \dfrac{r^2}{2} \partial_z  \Psi (r,h+\gamma_S(r)).
\end{eqnarray*}
Thus, on $\partial S_h \cap \partial \Omega_{h,1},$ there holds :
\begin{eqnarray*}
n \cdot \tilde{\varphi}_h &=&  \dfrac{r^2}{2} \partial_z\Psi(r,h+\gamma_S(r)) + \dfrac{\sqrt{1-r^2}}{r}  \partial_{r}\left( \dfrac{r^2}{2} \Psi\right)(r,h+\gamma_S(r)),\\
&=&\sqrt{1-r^2} = e_3 \cdot n\,.
\end{eqnarray*}
Before going to estimates on $\varphi_h$, we first obtain a precise decription of the size of  $\Psi$ and its derivatives.
Here and in what follows we use  $a \lesssim b$ to denote that $a\leq Cb$ for a constant $C$ independant of relevant parameters.
\begin{proposition} \label{prop_Psi}
The following inequalities hold true, for all values of $h \in (0,h_M)$
and   $(r,z) \in \Omega_{h,\delta}\,:$
\begin{eqnarray}
|\Psi(r,z)| &\lesssim& 1 \,,  \\
|\partial_z \Psi(r,z)| &\lesssim&  \dfrac{1}{h + \gamma_S(r)}\,, \qquad |\partial_r \Psi(r,z)| \lesssim  \dfrac{r}{h+\gamma_S(r)} \,,\qquad \label{eq_dPsi1} \\
|\partial_{rr} \Psi(r,z)|& \lesssim& \dfrac{1}{h+\gamma_S(r)} \,, \qquad |\partial_{zz} \Psi(r,z)| \lesssim \dfrac{1}{h + \gamma_S(r)}\,, \qquad\\
|\partial_{rz} \Psi(r,z) + \dfrac{r}{h +\gamma_S(r)} \partial_z \Psi(r,z)| &\lesssim&  \dfrac{r}{h+\gamma_S(r)} \,, \label{drzpsi} \qquad \\
|\partial_{zzz} \Psi(r,z)| + |\partial_{rrz} \Psi(r,z)| & \lesssim&  \dfrac{1}{(h+\gamma_S(r))^2}\,, \\
|\partial_{rzz} \Psi(r,z)|  \lesssim \dfrac{r}{(h+\gamma_S(r))^2}, & &
|\partial_{rrr} \Psi(r,z)|  \lesssim  \dfrac{r}{(h+\gamma_S(r))^2} + \frac{1}{h+\gamma_s(r)}  \label{eq_dPsifin}
\end{eqnarray}
and also
\begin{eqnarray}
|\partial_{h} \Psi(r,z)| & \lesssim & \dfrac{1}{h+ \gamma_S(r)}\,, \\
|\partial_{rh} \Psi(r,z)|  & \lesssim & \dfrac{1}{(h+\gamma_S(r))} + \dfrac{r}{(h+\gamma_S(r))^2}\,,\\
|\partial_{zzh}\Psi(r,z)| + |\partial_{rrh} \Psi(r,z)| + |\partial_{hz} \Psi(r,z)| & \lesssim & \dfrac{1}{(h+\gamma_S(r))^2}\,,\\
|\partial_{rzh} \Psi(r,z)| & \lesssim &\dfrac{1}{(h+\gamma_S(r))^2} +   \dfrac{r}{(h+\gamma_S(r))^3}\,.
\end{eqnarray}
\end{proposition}
\begin{proof}
At first, we note that
\beno
|\partial_r \alpha_S| + |\partial_r \alpha_P| &\lesssim&r , \\
|\partial_{rrr} \alpha_S| + |\partial_{rr} \alpha_S | + |\partial_{rrr} \alpha_P| +  |\partial_{rr} \alpha_{P}| &\lesssim& 1,
\eeno
This yields, by \eqref{def_Phideb}-\eqref{def_Phifin}, that we have:
$$
 \qquad |\partial_r \Phi(r,t) | +|\partial_{rtt} \Phi(t,r)|  + |\partial_{tr} \Phi(t,r)| \lesssim r  \,,
 \quad \forall \, (r,t) \in (0,\delta) \times (0,1)\,.
$$
and
$$
|\partial_t \Phi(r,t)| + |\partial_{rrr}\Phi(t,r)|+ |\partial_{rrt} \Phi(t,r)|+  |\partial_{rr} \Phi(t,r)| \lesssim 1\,.
$$
We remark also that
$$
 |\partial_{ttt} \Phi(t,r)| + |\partial_{tt} \Phi(t,r)| \lesssim (\alpha_P+\alpha_S) \lesssim (h+\gamma_S(r)).
$$
We introduce these bounds implicitly in the following computations.

\medskip

For instance, we have:
\begin{eqnarray*}
\partial_r \Psi(r,z)& =&  \partial_r \Phi\left(r , \dfrac{z}{h+\gamma_S(r)}\right) - \dfrac{z \gamma_S'(r)}{(h+\gamma_S(r))^2} \partial_t \Phi\left(r, \dfrac{z}{(h+\gamma_S(r)}\right)\\
			 & \lesssim &  \dfrac{r}{ (h+\gamma_S(r))}\,.
\end{eqnarray*}
We obtain similarly  \eqref{eq_dPsi1}--\eqref{eq_dPsifin} after tedious calculations.
The only computation we make precise is the one concerning $\partial_{rz} \Psi.$
Differentiating the last equality w.r.t. $z$, we get :
\begin{eqnarray*}
\partial_{rz} \Psi(r,z) &=& \dfrac{1}{h+\gamma_S(r)}\partial_{rt} \Phi  -  \dfrac{ \gamma_S'(r)}{(h+\gamma_S(r))^2} \partial_{t} \Phi
							-  \dfrac{z \gamma_S'(r)}{(h+\gamma_S(r))^3} \partial_{tt} \Phi,
\end{eqnarray*}
where :
$$
\partial_z \Psi(r,z) = \dfrac{1}{(h+\gamma_S(r))} \partial_{t} \Phi\,.
$$
Thus,
\begin{eqnarray*}
\partial_{rz} \Psi(r,z)  + \dfrac{r}{h +\gamma_S(r)} \partial_z \Psi(r,z) &=&
 \dfrac{1}{(h+\gamma_S(r))} \partial_{rt} \Phi
+ \dfrac{r  -  \gamma_S'(r)}{(h+\gamma_S(r))^2} \partial_{t} \Phi\\
&-& \dfrac{z \gamma_S'(r)}{(h+\gamma_S(r))^3} \partial_{tt} \Phi\,.
\end{eqnarray*}
where we dominate :
$$
\left| r-  \gamma_S'(r) \right| \lesssim r^3\,.
$$
Combining with the previous estimates on the derivative of $\Phi,$ we obtain the expected result.

\medskip
Concerning derivatives w.r.t. $h$ the computations are similar. First, we note that :
$$
|\partial_{r}^k \partial_h \alpha_S(r)| + |\partial_{r}^k \partial_h \alpha_P(r)| \lesssim 1 \,, \quad \forall \, k \in \mathbb N\,.
$$
Correspondingly, we obtain that
$$
| \partial^k_{r} \partial^l_t \partial_h \Phi | \lesssim 1 \,, \quad \forall (k,l) \in \mathbb N^2\,.
$$
Then, we obtain, for instance :
\begin{eqnarray*}
\partial_{rh} \Psi(r,z)& =&  \partial_{rh} \Phi  - \dfrac{z}{(h+\gamma_S(r))^2}  \partial_{rt} \Phi
						+ 2 \dfrac{z \gamma_S'(r)}{(h+\gamma_S(r))^3} \partial_t \Phi\\
						&-& \dfrac{z \gamma_S'(r)}{(h+\gamma_S(r))^2} \partial_{th}  \Phi
						+ \dfrac{z^2 \gamma_S'(r)}{(h+\gamma_S(r))^4} \partial_{tt}  \Phi
						\end{eqnarray*}
in which, we plug previous estimates on $\Phi$ and its derivative. This yields :
$$
|\partial_{rh} \Psi(r,z)| \lesssim  \dfrac{1}{(h+\gamma_S(r))} +  \dfrac{r}{(h+\gamma_S(r))^2}\,.
$$
Other estimates are obtained similarly.
\end{proof}

\medskip

The rest of  this appendix is devoted to the proofs of lemmas \ref{est:1} and \ref{est:2}.
They rely on the following lemma whose proof can be found in \cite{DGVH10}:
\begin{lemma}\label{lem:int}
Given $(p, q)\in (0,\infty)^2$, when $h$ goes to 0, the quantity
\beno
\int_0^\delta \f{|x|^pdx}{(h+|x|^{2})^q}
\eeno
behaves like\\
(1) $h^{\f{p+1}{2}-q}$, if $p+1< 2 q $,\\
(2) $\ln{h}$, if $p+1= 2 q,$\\
(3) $C$, if $p+1> 2 q$,\\
with $C$ a constant depending on $(p, q)$.
\end{lemma}

\subsection{Proof of Lemma \ref{est:1}}
First, we recall that, by construction $\varphi_h$ is smooth outside $\Omega_{h,\delta}.$ Thus we have get that
\beno
\|\nabla^k \varphi_h\|_{L^\infty(F_h \setminus \Omega_{h,\delta})}\lesssim 1 \,, \quad \forall \, k \in \mathbb N\,.
\eeno
Hence, we focus on $\Omega_{h,\delta}$ and replace $\varphi_h$ with $\widetilde{\varphi}_h$ in what remains.
We keep notations of the previous lemma. In particular, we have by definition :
$$
\widetilde{\varphi}_h = \dfrac{1}{2} \left( -\partial_z \Psi(r,z) r e_r + \dfrac{1}{r}\partial_r\left[ r^2 \Psi(r,z)\right]e_{z}\right), \quad \text{ in $\Omega_{h,\delta}.$}
$$

\paragraph{\em Sobolev bounds.}
First, applying the bounds of the previous lemma, we get:
\beno
\|\widetilde{\varphi}_h\|_{L^2(\Omega_{h,\delta})} &\lesssim &\|r \partial_z \Psi\|_{L^2(\Omega_{h,\delta})} +  \|\dfrac{1}{r}\partial_r\left[ r^2 \Psi(r,z)\right]\|_{L^2(\Omega_{h,\delta})}\\
&\lesssim&  \left[  1 + \int^\delta_0\f{r^3{\rm d}r}{(r^2+h)}\right]^{\frac{1}{2}}\lesssim 1.
\eeno
Then, in cylindrical coordinates, we obtain:
\begin{equation} \label{est:nablaphipct}
|\nabla \widetilde{\varphi}_h| \lesssim |\partial_z \Psi| + |r \partial_{zz} \Psi| + |\partial_r \Psi| +  |r\partial_{rr} \Psi| + \frac{r^2}{h+\gamma_s(r)},
\end{equation}
where we used \eqref{drzpsi} to get rid of $\partial_{rz} \Psi$ terms. Consequently :
$$
\|\nabla \widetilde{\varphi}_h\|_{L^2(\Omega_{h,\delta})}^2 \lesssim  1+ \int_{0}^{\delta} \dfrac{r {\rm d}r}{(h+\gamma_S(r))}\, \lesssim |\ln(h)|\,.
$$
On the other hand,  we have
\beno
\|\pa_2\varphi_h^2\|^2_{L^2(\Omega_{h,\delta})} &=& \dfrac{1}{4}\left\| \left(1 - \dfrac{r^2 |\sin(\theta)|^2}{(h+ \gamma_S(r))^2}\right) \partial_z \Psi(r,z) \right\|^2_{L^2(\Omega_{h,\delta})}  \: +  \: O(1)
\eeno
Let now $\varepsilon >0$ small enough. There holds:
\beno
\|\pa_2\varphi_h^2\|^2_{L^2(\Omega_{h,\delta})} &\geq & \frac{1}{8} \left\|  \partial_z \Psi(r,z) \right\|^2_{L^2(\Omega_{h,\delta}) \cap \{|\theta| < \varepsilon\} }  \: +  \: O(1) \\
														& \geq & c \int_{0}^1 \int_{0}^{\delta} |\partial_t \Phi(r,t)|^2 \dfrac{r{\rm d}r \, {\rm d}t}{(h+\gamma_S(r))} \: +  \: O(1).
\eeno
where $c=c_\eps$ is a small positive constant. Finally, as $\pa_t \Phi$ is bounded from below, we get
$$
\|\pa_2\varphi_h^2\|^2_{L^2(\Omega_{h,\delta})}  \geq   c' \int_{0}^1 \int_{0}^{\delta}  \dfrac{r{\rm d}r \, {\rm d}t}{(h+\gamma_S(r))} + O(1)   \geq c |\ln h|.
$$
Combining the above estimate yields :
\beno
 c|\ln{h}|\leq \|D(\varphi_F)\|^2_{L^2(F(t))}\leq \|\nabla \varphi_F\|^2_{L^2(F(t))}\leq C|\ln{h}|.
\eeno
This completes the proof of \eqref{est:L2}-\eqref{est:gradL2_2}.

\paragraph{\em Boundary estimates.} We proceed with estimate \eqref{est:bdy}. By standard identities of differential geometry (see \cite{BDG}), we have
\beno
2D(\widetilde{\varphi}_h)n\times n = {\pa_n \widetilde{\varphi}_h} \times n+ (\widetilde{\varphi}_h-e_3) \times n.
\eeno
Thus, we only need to prove that
\beno
\|\beta_S {\pa_n \widetilde{\varphi}_h} \times n + (\beta_S+1)\widetilde{\varphi}_h\times n \|_{L^2(\pa S_h \, \cap \, \partial \Omega_{h,2\delta})}\leq C.
\eeno
In the aperture $\Omega_{h,2\delta}$ we recall that the normal to $\partial S_h$ directed toward $S_h$ satisfies
$n = -re_r + ((1+h)-z) e_z.$ This yields:
\begin{eqnarray*}
2\widetilde{\varphi}_h  \times n  	&=& \left( - \partial_z \Psi\ r e_r   + \dfrac{1}{r} \partial_r \left[ r^2 \Psi \right] e_z \right) \times ( -re_r + ((1+h)-z) e_z )\,,\\
			&=&  \left( r ((1+h)-z) \partial_z \Psi -      \partial_r \left[ r^2 \Psi \right]  \right) e_{\theta}\,.
\end{eqnarray*}
Note here that, on $\partial S_h,$ we have $(1+h)-z - 1 \lesssim r^2$ and that
$$
\partial_r \left[ r^2 \Psi \right] \lesssim r = O(1)\,.
$$
Here and in the end of this proof $O(1)$ means bounded in $L^2(\partial S_h \, \cap \, \partial \Omega_{h,2\delta} )\,.$
This yields:
$$
\widetilde{\varphi}_h  \times n  = \dfrac{r}{2}    \partial_z \Psi e_{\theta} + O(1) \,.
$$
On the other hand, in the aperture $\Omega_{h,\delta}$ we have:
\begin{eqnarray*}
2\, \partial_n \widetilde{\varphi}_h &=& \left( r \partial_{rz} (\Psi r)  - ((1+h)-z) r \partial_{zz}\Psi\right) e_r \\
				&+&\left(  -r \partial_r \left[\frac{1}{r} \partial_r[r^2\Psi]\right] +  ((1+h)-z)\partial_z \left[ \frac{1}{r} \partial_r [ r^2 \Psi ] \right] \right)e_z.
\end{eqnarray*}
Introducing the estimates of Proposition \ref{prop_Psi} we get in particular:

\begin{eqnarray*}
\partial_z \left[ \frac{1}{r} \partial_r [ r^2 \Psi ] \right] &=& 2 \partial_z \Psi + r \partial_{rz} \Psi \\
										&=& \left(2 - \dfrac{r^2}{h+\gamma_S(r)}  \right) \partial_z \Psi + O(1)=   \frac{2h}{h + \gamma_s(r)} \partial_z \Psi + O(1).
\end{eqnarray*}
Similarly, we have:
$$
r \partial_{rz} (\Psi r)  = r^2 \partial_{rz} \Psi + r \partial_z \Psi  = - r \partial_z \Psi + r \left( -\frac{r^2}{h + \gamma_s(r)}  + 2\right) \pa_z \Psi  + O(1) = - r \pa_z \Psi + O(1),
$$
using that
\begin{equation} \label{simplif}
 r \left( -\frac{r^2}{h + \gamma_s(r)}  + 2\right) \pa_z \Psi = O\left( \frac{r h}{(h + \gamma_s(r))^2} \right)
 \end{equation}
is uniformly bounded in $L^2$ (use case 3 of Lemma \ref{lem:int} with $p=3$, $q=4$). Then,
$$
2 \partial_n \widetilde{\varphi}_h =  -( r \partial_z \Psi +  r \partial_{zz}\Psi) e_r \:   + \:   \frac{2h}{h+\gamma_s(r)} \pa_z \Psi e_z \: + \:  O(1).
$$
Recalling that  $n = -re_r + ((1+h)-z) e_z$, and still applying \eqref{simplif}, we obtain:
$$
\partial_{n} \widetilde{\varphi}_h \times n = \dfrac{r}{2} \left(   \partial_{zz}\Psi + \partial_z \Psi\right) e_{\theta}   + O(1).
$$
Finally, we obtain:
$$
\beta_S \partial_n \widetilde{\varphi}_h  \times n + (\beta_S + 1) \tilde{\varphi}_h \times n = \dfrac{r\beta_S}{2} \left(    \partial_{zz} \Psi  + \left(\dfrac{1}{\beta_S} + 2 \right) \partial_z \Psi\right) + O(1)\,.
$$
By our choice of $\alpha_S$, it follows that (see \eqref{prop_Philast})
\begin{equation*}
\partial_{zz} \Psi(r,h+\gamma_S(r))  + \left(\dfrac{1}{\beta_S} + 2 \right) \partial_z \Psi(r,h+\gamma_S(r))  = O(1).
 \end{equation*}
This yields \eqref{est:bdy}.

\medskip

\paragraph{\em $h$-derivatives.} For $\pa_h \widetilde{\varphi}_h$, we note that, for $i=1,2,3$ we have:
\beno
|\pa_h \widetilde{\varphi}^i_h|\lesssim |r\partial_{zh} \Psi|  + |\partial_h \Psi| + |r \partial_{rh} \Psi| \lesssim \dfrac{r}{(h+\gamma_S(r))^2} + \dfrac{1}{(h+\gamma_S(r))}\,.
\eeno
Consequently, as integrating with respect to the $z$-variable reduces to remove one power of $(h+\gamma_S(r))$ in
the denominator, we get
$$
\|\int^{h+\gamma_s(r)}_z\pa_h\widetilde{\varphi}^i_Fds\|^2_{L^2(F_h \cap \Omega_{h,\delta})} \lesssim 1+ \int_{0}^{\delta} \f{r^3 {\rm d}r}{(h + \gamma_S(r))} \lesssim 1\,.
$$
and
$$
\|\int^{h+\gamma_s(r)}_0\pa_h\widetilde{\varphi}^i_Fdz\|^2_{L^2(\partial \Omega \cap \partial \Omega_{h,\delta}) }
\lesssim 1 +  \int^\delta_0\f{r^3{\rm d}r}{(h + \gamma_S(r))^2} \lesssim C|\ln{h}|\,.
$$
Finally, we also compute: for $i=1,2$:
\beno
|\pa_i \pa_h \widetilde{\varphi}_h| &\lesssim& |\partial_{zh} \Psi| + |r \partial_{rzh} \Psi| + |\partial_{rh} \Psi| + |r \partial_{rrh} \Psi| \\
							& \lesssim & \dfrac{1}{(h+\gamma_S(r))^2}
\eeno
so that for the same reason as above, we have:
\beno
|(h+\gamma_s(r))\pa_i\int^{h+\gamma_s(r)}_z\pa_h\widetilde{\varphi}_h^i(x_1,x_2,s)ds| \lesssim 1 \,, \quad \text{ for $i=1,2$}\, \\
\eeno
and
$$
\|(h+\gamma_s(r))\pa_i\int^{h+\gamma_s(r)}_z\pa_h\widetilde{\varphi}_F^i(x_1,x_2,s)ds\|_{L^2(F_h\cap \Omega_{h,\delta})}\lesssim 1,\quad \text{ for $i=1,2,3$}\, .
$$
This ends the proof of \eqref{dhpsi1_2}--\eqref{dhpsi3_2}.

\paragraph{\em Pointwise bounds.}
We easily get that on $\Omega_{h,\delta}$
\beno
|\nabla \varphi_h| \lesssim |r\partial_{rz}\Psi| + |\partial_z \Psi| + |r\partial_{rr} \Psi| + |\partial_{r}\Psi| + |r\partial_{zz}\Psi| \lesssim \dfrac{1}{h + \gamma_S(r)},
\eeno
from which we conclude:
$$
\left|\int^{h+\gamma_S(r)}_z(\pa_i\varphi_h^j(x_1,x_2,s)+\pa_j\varphi_h^i(x_1,x_2,s))ds\right|\lesssim 1, \text{ for $(i,j) \in \{1,2,3\}^2.$}
$$
Similarly, we get:
\beno
|\partial_{1} \nabla \varphi_h| + |\partial_2 \nabla \varphi_h| \lesssim |r \partial_{rrr}\Psi| + |\partial_{rr} \Psi| + |r \partial_{rrz} \Psi| + |\partial_{rz} \Psi| + |  \pa_{zz} \Psi | \lesssim \, \dfrac{r}{(h+\gamma_S(r))^2} \: + \: \frac{1}{h+\gamma_s}.
 \eeno
and we conclude:
$$
\left|(h+\gamma_S(r))\pa_i\int^{h+\gamma_S(r)}_z(\pa_i\varphi_F^j(x_1,x_2,s)+\pa_j\varphi_F^i(x_1,x_2,s))ds\right| \lesssim 1.
$$
This ends the proof of our lemma.


\subsection{Proof of Lemma \ref{est:2}}
\paragraph{\em Boundary  integrals} We focus on the domain $\Omega_{h,\delta}$ and we replace $\varphi_h$ with $\widetilde{\varphi}_h$
in the proof because everything is bounded elsewhere. First, it turns out that the first identity is independant of $q_h.$
Indeed, applying boundary conditions
\eqref{bdy1} we obtain :
\beno
\int_{\partial \Omega_{h,\delta} \cap \pa \Omega } (2 D(\widetilde{\varphi}_h) - q_h) n \cdot \widetilde{\varphi}_h \text{d$\sigma$}
&=&
\int_{\partial \Omega_{h,\delta} \cap \pa \Omega } [(2 D(\widetilde{\varphi}_h) n) \times n]   \cdot ( \widetilde{\varphi}_h \times n) \text{d$\sigma$}\\
&=&
-\dfrac{1}{\beta_{\Omega}} \int_{\pa \Omega_{h,\delta} \cap \pa \Omega} |\widetilde{\varphi}_h \times n|^2 \text{d$\sigma$}\,.
\eeno
Similarly, introducing boundary conditions \eqref{bdy1} and estimate \eqref{est:bdy}, we compute:
\begin{multline*}
\int_{\partial \Omega_{h,\delta} \cap \partial S_h } (2 D(\widetilde{\varphi}_h) - q_h) n \cdot (e_3 - \widetilde{\varphi}_h) \text{d$\sigma$}
=
\int_{\partial \Omega_{h,\delta} \cap \partial S_h } [(2 D(\widetilde{\varphi}_h) n) \times n]   \cdot ( (e_3 -  \widetilde{\varphi}_h) \times n) \text{d$\sigma$}\\
=
\dfrac{1}{\beta_{S}} \int_{\pa \Omega_{h,\delta} \cap \partial S_h} |( e_3 - \widetilde{\varphi}_h) \times n|^2 \text{d$\sigma$} + O(\|( e_3 - \widetilde{\varphi}_h) \times n\|_{L^2(\partial S_h \cap \Omega_{h,\delta})})\,.
\end{multline*}
On the boundaries, there holds
\beno
\widetilde{\varphi}_h =  - \dfrac{P_1(r)}{(h+\gamma_S(r))}re_r \quad \textrm{in} \quad \pa \Omega_{h,\delta}\cap \pa \Omega,
\quad
|\widetilde{\varphi}_h|  \lesssim \dfrac{r}{(h+\gamma_S(r))}   \quad \textrm{in} \quad \pa \Omega_{h,\delta}\cap \pa S_h,
\eeno
which implies that
\beno
|\ln(h)| \lesssim \int^\delta_0 \f{r^3{dr}}{(h+\gamma_S(r))^2} \lesssim \|\widetilde{\varphi}_h \times n \|^2_{L^2(\pa \Omega_{h,\delta}\cap \pa \Omega)} =  \|\widetilde{\varphi}_h \|^2_{L^2(\pa \Omega_{h,\delta}\cap \pa \Omega)}  \eeno
and
which implies that
\beno
|\ln{h}|\lesssim \int_{\partial \Omega_{h,\delta} \cap \partial S_h } (2 D(\widetilde{\varphi}_h) - q_h) n \cdot (e_3 - \widetilde{\varphi}_h) \text{d$\sigma$}
  - \int_{\pa \Omega \cap \pa\Omega_{h,\delta}}(D(\widetilde{\varphi}_h)-q_hI)n\cdot \varphi_h d\sigma
 \lesssim |\ln(h)|.
 \eeno

\paragraph{\em Construction of $q_h$.}
Again, we focus on the domain $\Omega_{h,\delta}$  where $\varphi_h=\widetilde{\varphi}_h$ to compute $q_h.$
The construction is extended to the whole $F_h$ by a standard truncation argument.

\medskip

Given $h>0,$ let define
\begin{equation}\label{q}
2q_h \triangleq - \left( r \partial_{rz} \Psi + 2 \partial_z \Psi + \int^r_0\pa_{zzz} \Psi(s,z)sds \right). \nonumber\\
\end{equation}
Obviously, the map $(h,x) \mapsto q_h(x)$ is $C^1$ on $\{ h \in (0,h_M), \; x \in F_h\}.$
and we have $\Delta \widetilde{\varphi}_h - \nabla q_h = f_h e_3$ where
$$
|f_h| \lesssim |r\partial_{rrr} \Psi| + |\partial_{rr} \Psi| + \left|\dfrac{\partial_r \Psi}{r} \right| + |r \partial_{rzz}\Psi| + |\partial_{zz} \Psi| \lesssim \dfrac{1}{(h+\gamma_s(r))}
$$
Then, for arbitrary $h \in (0,h_M)$ and $v \in H^1(F_h)$ we introduce $\chi$ the truncation function as defined in  \eqref{eq_chi} and
 we split :
 $$
  \int_{F_h}(\Delta \varphi_h-\nabla q_h)v  = \int_{F_h} \chi (\Delta \varphi_h-\nabla q_h)v + \int_{F_h} (1-\chi) (\Delta \varphi_h-\nabla q_h)v  = I + II.
 $$
Then, we have:
\beno
\left| I  \right| &=& \int_{F_h \cap \Omega_{h,2\delta}} |\chi f_h|| v_3| \\
&\leq & \| (h+\gamma_S) f_h\|_{L^2(F_h \cap \Omega_{h,2\delta})} \| \f{v_3}{\gamma_s(r)+h} \|_{L^2(F_h \cap \Omega_{h,2\delta})}
\eeno
where a standard Poincar\'e inequality yields:
\beno
\|\f{v_3}{\gamma_s(r)+h}\|_{L^2(F_h \cap \Omega_{h,2\delta})}\lesssim \|\pa_3v_3\|_{L^2(F_h \cap \Omega_{h,\delta})} \lesssim  \|D(v)\|_{L^2(F_h \cap \Omega_{h,\delta})}.
\eeno
As for the second integral, we introduce $P_z R^i$ for $i=1,2,3,$ which satisfy :
\begin{eqnarray}
P_z R^i &=& \int_z^{h+\gamma_s(r)} (1-\chi) (\Delta \varphi^i_h-\pa_i q_h) \quad \text{for $r < 1$, $\: z < r + \gamma_s(r)$}\,\\
P_z R^i &=& \int_z^{h+1}(1-\chi) (\Delta \varphi^i_h-\pa_i q_h)   \quad \text{ otherwise}.
\end{eqnarray}
Repeating the computations of $I_1(t)$ and $I_2(t)$ in Section \ref{sec_prfslip}, we obtain (the summation over index
$i$ is implicit) :
\beno
| II | =  \int_{F_h} (\partial_3 v^i + \partial_i v^3 ) P_z R^i + \int_{F_h}  v^3 \partial_i P_z R^i - \int_{\pa \Omega} (v^3 n^i + v^in^3) P_z R^i \text{d$\sigma$}\,.
\eeno
As $(\Delta \varphi_h-\nabla  q_h)$ is uniformly bounded outside $\Omega_{h,\delta},$
applying a Cauchy Schwarz inequality to dominate $K_2$ and combining with $K_1$ concludes the proof :
$$
\left| \int_{F_h}(\Delta \varphi_h-\nabla q_h) \cdot v \right|  \lesssim  \|D(v)\|_{L^2(F_h)} + \|v\|_{L^2(\pa \Omega)}.
$$

\section{Estimates on the test function in the mixed case}\label{app_2}
In this appendix, we collect estimates on the test function of Section \ref{sec_mix}.
As previously, we define,
$$
\Psi(r,z) := \Phi\left(r, \dfrac{z}{h+\gamma_S(r)}\right)\,, \quad \forall  \, h \in (0,h_M),  \quad \forall \,(r,z) \in \Omega_{h,1}\,.
$$
The only point which differs from the test-function in the slip case is the decription of this $\Psi.$
Combining these estimates in the same manner as in the previous appendix, we obtain Lemma \ref{est_m:1}.
As the computations are completely similar, all the proofs are left to the reader.
\begin{proposition} \label{prop_Psi_mix}
The following inequalities hold true, for all  $h \in (0,h_M)$
and   $(r,z) \in \Omega_{h,\delta}\,:$
\begin{eqnarray}
|\Psi(r,z)| &\lesssim& 1 \,,  \\
|\partial_z \Psi(r,z)| &\lesssim&  \dfrac{1}{h + \gamma_S(r)}\,, \qquad |\partial_r \Psi(r,z)| \lesssim  \dfrac{r}{h+\gamma_S(r)} \,,\qquad  \\
|\partial_{rr} \Psi(r,z)|& \lesssim& \dfrac{1}{h+\gamma_S(r)} \,, \qquad |\partial_{zz} \Psi(r,z)| \lesssim \dfrac{1}{(h + \gamma_S(r)^2}\,, \qquad\\
|\partial_{rz} \Psi(r,z)| &\lesssim&  \dfrac{r}{(h+\gamma_S(r))^2} \,, \qquad |\partial_{zzz} \Psi(r,z)|  \lesssim  \dfrac{1}{(h+\gamma_S(r))^3}  \\
|\partial_{rrz} \Psi(r,z)| & \lesssim&  \dfrac{1}{(h+\gamma_S(r))^2}\,, \qquad |\partial_{rzz} \Psi(r,z)|  \lesssim \dfrac{r}{(h+\gamma_S(r))^3}\,, \\
|\partial_{rrr} \Psi(r,z)| & \lesssim&  \dfrac{r}{(h+\gamma_S(r))^2}\,,\qquad |\partial_{rrrz} \Psi(r,z)|  \lesssim \dfrac{r}{(h+\gamma_S(r))^3}\,, \\
\end{eqnarray}
and also
\begin{eqnarray}
|\partial_{h} \Psi(r,z)| & \lesssim & \dfrac{1}{h+ \gamma_S(r)}\,, \qquad
|\partial_{rh} \Psi(r,z)|   \lesssim  \dfrac{r}{(h+\gamma_S(r))^2}\,,\\
|\partial_{rrh} \Psi(r,z)|  & \lesssim & \dfrac{1}{(h+\gamma_S(r))^2}\,,
\qquad  |\partial_{hz} \Psi(r,z)|  \lesssim  \dfrac{1}{(h+\gamma_S(r))^2} \\
|\partial_{rzh} \Psi(r,z)| & \lesssim &  \dfrac{r}{(h+\gamma_S(r))^3} + \dfrac{1}{(h+\gamma_S(r))^2} \\
|\partial_{zzh}\Psi(r,z)| &\leq& \dfrac{1}{(h+\gamma_S(r))^3} \,.
\end{eqnarray}
\end{proposition}

We also need a different way to construct the pressure $q_h$ in order to prove Lemma \ref{est_m:2}. Again, we construct $q_h$
only in the aperture domain $\Omega_{h,\delta},$ the extension to $F_h$ being done by a standard
truncation. So, for $h>0$ we set :
$$
2 q_h \triangleq  r \partial_{rz} \Psi + 2 \partial_z \Psi - \int_0^r \partial_{zzz} \Psi s \text{d$s$}\,.
$$
In this way, we obtain $\Delta {\varphi}_h - \nabla q_h =   f^r e_r +  f^z e_z$ where the $f_i$'s are bounded outside
$\Omega_{h,\delta}$ and, in $\Omega_{h,\delta},$ we have:
\begin{eqnarray}
|f^z| &\lesssim& |r \partial_{rrr} \Psi| + |\partial_{rr} \Psi| + \left|  \dfrac{\partial_r \Psi}{r} \right| \lesssim \dfrac{1}{(h+\gamma_S(r))} \label{f3}\\
|f^r| &\lesssim & |r \partial_{rrz} \Psi| + |\partial_{rz} \Psi| \lesssim \dfrac{r}{(h+\gamma_S(r))^2}				\label{fp1}\\
|\partial_r f^r| + \left| \dfrac{f^r}{r} \right|&\lesssim&  |r\partial_{rrrz} \Psi| +  |\partial_{rrz} \Psi| + \left| \dfrac{\partial_{rz}\Psi}{r}\right| \lesssim \dfrac{1}{(h+\gamma_S(r))^2}. \label{fp2}
\end{eqnarray}
Then, for any $v \in H^1(\Omega)$ such that $v \cdot n = 0$ on $\pa \Omega$ and $v \in \mathcal R$ on $S_h$ we have :
$$
\int_{F_h} (\Delta \varphi_h - \nabla q_h) \cdot v = \int_{F_h} f^z v^3 + \int_{F_h}  f^r e_r \cdot v .
$$
First, we apply  Poincar\'e inequality, as $v^3 = 0$ on $\pa \Omega \cap \pa \Omega_{h,\delta},$ together with \eqref{f3}. This yields:
\beno
\left| \int_{F_h} f^z v_3 \right| &\lesssim&  \left| \int_{F_h \cap \Omega_{h,\delta}} f^z (h+\gamma_S(r)) \dfrac{v_3}{(h+\gamma_S(r))}\right| +\|v\|_{L^2(F_h)}  \\
						&\lesssim & \|D(v)\|_{L^2(F_h)} + \|v\|_{L^2(F_h)}
\eeno
For the second integral:  we write $f^r e_r  = f^1 e_1 + f^2 e_2$, and split the integral, with $\chi$ as defined in \eqref{eq_chi} :
$$
\int_{F_h}  f^i v_i =  \int_{F_h} \chi  f^i   v^i + O(\|v\|_{L^2(F_h)})\,,
$$
as  $f^r$ is uniformly bounded outside $\Omega_{h,\delta}.$
Then, we introduce  $P_z f^i,$ for $i =1,2,$ which satisfies:
$$
P_z f^i  = \int_0^z \chi f^i
$$
Through integration by parts, we get (with the Einstein convention of repeated indices):
\beno
 \int_{F_h} \chi f^i v^i &=& \int_{F_h \cap \Omega_{h,2\delta}} P_z f^i (\partial_3 v^i + \partial_i v^3) + \int_{F_h \cap \Omega_{h,2\delta}}  \dfrac{v_3}{h+\gamma_S(r)} (h+\gamma_S(r)) \left( \pa_i P_z f^i  \right) \, \\
 				 	&& - \int_{\partial S_h \cap \Omega_{h,2\delta}} P_z f^i  (v^i n^3 + v^3 n^i) \: \triangleq \: I + II + III
\eeno
Thanks to the bounds \eqref{fp1}-\eqref{fp2}, we get:
$$ |I| + |II| \: \lesssim \: \| D(v) \|_{L^2(F_h)}. $$
As regards the last integral, we simply write, because $v_{|_{S_h}} \in {\cal R},$
$$ | III | \: \lesssim \: \| P_z f^i \|_{L^1(\pa S_h)} \, \| v \|_{L^\infty(\pa S_h)}  \: \lesssim \:  \| v \|_{L^2(S_h)} \lesssim  \| v \|_{L^2(\Omega)} $$
This ends the proof.


\begin{thebibliography}{99}


\bibitem{BarnockyDavis89}
G.~Barnocky and R.~H. Davis.
\newblock The influence of pressure-dependent density and viscosity on the
  elastohydrodynamic collision and rebound of two spheres.
\newblock {\em J. Fluid Mech.}, 209:501--519, 1989.

\bibitem{Bocquet}
L.~Bocquet and J.~Barrat.
\newblock Flow boundary conditions from nano-to micro-scales.
\newblock {\em Soft Matter}, 3:985--693, 2007.

\bibitem{BDG} D. Bresch, B. Desjardins, and D. Gerard-Varet. On compressible Navier-Stokes equations with
density dependent viscosities in bounded domains, {\it J. Math. Pures Appl.(9)}, 87(2) : 227--235.
2007.


\bibitem{Conca&al00}
C.~Conca, J.~A. San~Mart{\'{\i}}n, and M.~Tucsnak.
\newblock Existence of solutions for the equations modelling the motion of a
  rigid body in a viscous fluid.
\newblock {\em Comm. Partial Differential Equations}, 25(5-6):1019--1042, 2000.


\bibitem{Cooley&Oneill69}
M.~Cooley and M.~O'Neill.
\newblock On the slow motion generated in a viscous fluid by the approach of a
  sphere to a plane wall or stationary sphere.
\newblock {\em Mathematika}, 16:37--49, 1969.



\bibitem{Cox74}
R.G. Cox.
\newblock The motion of suspended particles almost in contact.
\newblock{ \em Int. J. Multiphase Flow}, 1:343--371,1974.


\bibitem{CumsilleTakahashi}
P.~Cumsille and T.~Takahashi.
\newblock Wellposedness for the system modelling the motion of a rigid body of
  arbitrary form in an incompressible viscous fluid.
\newblock {\em Czechoslovak Math. J.}, 58(133)(4):961--992, 2008.

\bibitem{davisetal86}
R.~H. Davis, J.-M. Serayssol, and E.~Hinch.
\newblock The elastohydrodynamic collision of two spheres.
\newblock {\em J. Fluid Mech.}, 163:479--497, 1986.

\bibitem{Desjardins&Esteban99}
B.~Desjardins and M.~J. Esteban.
\newblock Existence of weak solutions for the motion of rigid bodies in a
  viscous fluid.
\newblock {\em Arch. Ration. Mech. Anal.}, 146(1):59--71, 1999.

\bibitem{Desjardins&Esteban00}
B.~Desjardins and M.~J. Esteban.
\newblock On weak solutions for fluid-rigid structure interaction: compressible
  and incompressible models.
\newblock {\em Comm. Partial Differential Equations}, 25(7-8):1399--1413, 2000.

\bibitem{Feireisl03}
E.~Feireisl.
\newblock On the motion of rigid bodies in a viscous incompressible fluid.
\newblock {\em J. Evol. Equ.}, 3(3):419--441, 2003.
\newblock Dedicated to Philippe B\'enilan.

\bibitem{GaldiSilvestre05}
G.~P. Galdi and A.~L. Silvestre.
\newblock Strong solutions to the {N}avier-{S}tokes equations around a rotating
  obstacle.
\newblock {\em Arch. Ration. Mech. Anal.}, 176(3):331--350, 2005.

\bibitem{DGVH10}
D.~G{\'e}rard-Varet and M.~Hillairet.
\newblock Regularity issues in the problem of fluid structure interaction.
\newblock {\em Arch. Ration. Mech. Anal.}, 195(2):375--407, 2010.

\bibitem{DGVH12}
D.~G{\'e}rard-Varet and M.~Hillairet.
\newblock Computation of the drag force on a sphere close to a wall: the
  roughness issue.
\newblock {\em ESAIM Math. Model. Numer. Anal.}, 46(5):1201--1224, 2012.

\bibitem{DGVHpp}
D.~G{\'e}rard-Varet and M.~Hillairet.
\newblock Existence of weak solutions up to collision for viscous fluid-solid
  systems with slip.
\newblock To appear in {\em Comm. Pure Appl. Math.},
\newblock available at  \url{http://hal.archives-ouvertes.fr/hal-00713331}, 2012.

\bibitem{GlassSueurpp}
O.~Glass and F.~Sueur.
\newblock Uniqueness results for weak solutions of two-dimensional fluid-solid
  systems.
\newblock arXiv:1203.2894v1, March 2012.

\bibitem{Grandmont&Maday98}
C.~Grandmont and Y.~Maday.
\newblock Existence de solutions d'un probl\`eme de couplage fluide-structure
  bidimensionnel instationnarie.
\newblock {\em C. R. Acad. Sci. Paris S\'er. I Math.}, 326(4):525--530, 1998.

\bibitem{Gunzburger&Lee&Seregin00}
M.~D. Gunzburger, H.-C. Lee, and G.~A. Seregin.
\newblock Global existence of weak solutions for viscous incompressible flows
  around a moving rigid body in three dimensions.
\newblock {\em J. Math. Fluid Mech.}, 2(3):219--266, 2000.

\bibitem{Hillairet07b}
M.~Hillairet.
\newblock Lack of collision between solid bodies in a 2{D} incompressible
  viscous flow.
\newblock {\em Comm. Partial Differential Equations}, 32(7-9):1345--1371, 2007.

\bibitem{HLS11}
M.~Hillairet, A.~Lozinski, and M.~Szopos.
\newblock Simulation of particulate flow governed by lubrication forces and
  far-field hydrodynamic interactions,.
\newblock {\em Discrete and Continuous Dynamical Systems, Series B}, 11:935 --
  956, 2011.

\bibitem{HillairetTakahashi09}
M.~Hillairet and T.~Takahashi.
\newblock Collisions in three-dimensional fluid structure interaction problems.
\newblock {\em SIAM J. Math. Anal.}, 40(6):2451--2477, 2009.

\bibitem{HillairetTakahashi10}
M.~Hillairet and T.~Takahashi.
\newblock Blow up and grazing collision in viscous fluid solid interaction
  systems.
\newblock {\em Ann. Inst. H. Poincar\'e Anal. Non Lin\'eaire}, 27(1):291--313,
  2010.


\bibitem{Hocking}
L.~Hocking.
\newblock The effect of slip on the motion of a sphere close to a wall and of
  two adjacent sheres.
\newblock {\em J. Eng. Mech.}, 7:207--221, 1973.

\bibitem{Hoffmann&Starovoitov99}
K.-H. Hoffmann and V.~Starovoitov.
\newblock On a motion of a solid body in a viscous fluid. {T}wo-dimensional
  case.
\newblock {\em Adv. Math. Sci. Appl.}, 9(2):633--648, 1999.

\bibitem{Hoffmann&Starovoitov00}
K.-H. Hoffmann and V.~Starovoitov.
\newblock Zur {B}ewegung einer {K}ugel in einer z\"ahen {F}l\"ussigkeit.
\newblock {\em Doc. Math.}, 5:15--21 (electronic), 2000.

\bibitem{JosephPhD}
G.~Joseph.
\newblock {\em Collisional dynamics of macroscopic particles in a viscous
  fluid}.
\newblock PhD thesis, California Institute of Technology, Pasadena, California,
  May 2003.

\bibitem{KytomaaSchmid92}
H.~Kyt\"omaa and P.~Schmid.
\newblock On the collision of rigid spheres in a weakly compressible fluid.
\newblock {\em Phys. Fluids A}, 4:2683--2689, 1992.

\bibitem{Lauga}
E.~Lauga, M.~Brenner, and H.~Stone.
\newblock Microfluidics: The no-slip boundary condition.
\newblock 2007.


\bibitem{SanMartin&al02}
J.~San~Mart{\'{\i}}n, V.~Starovoitov, and M.~Tucsnak.
\newblock Global weak solutions for the two-dimensional motion of several rigid
  bodies in an incompressible viscous fluid.
\newblock {\em Arch. Ration. Mech. Anal.}, 161(2):113--147, 2002.

\bibitem{Serre87}
D.~Serre.
\newblock Chute libre d'un solide dans un fluide visqueux incompressible.
  {E}xistence.
\newblock {\em Japan J. Appl. Math.}, 4(1):99--110, 1987.

\bibitem{SueurPlanaspp}
F.~Sueur and G.~Planas.
\newblock On the inviscid limit of the system "viscous incompressible fluid + rigid body" with the Navier conditions.
\newblock  arXiv:1206.0029, May 2012.

\bibitem{Takahashi03bis}
T.~Takahashi.
\newblock Analysis of strong solutions for the equations modeling the motion of
  a rigid-fluid system in a bounded domain.
\newblock {\em Adv. Differential Equations}, 8(12):1499--1532, 2003.

\bibitem{Takahashi03}
T.~Takahashi.
\newblock Existence of strong solutions for the problem of a rigid-fluid
  system.
\newblock {\em C. R. Math. Acad. Sci. Paris}, 336(5):453--458, 2003.

\bibitem{Takahashi&Tucsnak04}
T.~Takahashi and M.~Tucsnak.
\newblock Global strong solutions for the two-dimensional motion of an infinite
  cylinder in a viscous fluid.
\newblock {\em J. Math. Fluid Mech.}, 6(1):53--77, 2004.


\end{thebibliography}
\end{document}